\documentclass[reqno]{amsart}

\usepackage{amsmath,amssymb,graphicx%,mathabx
}
\usepackage{verbatim} 

\usepackage[latin1]{inputenc}
\usepackage{enumerate}
\usepackage{hyperref}

\graphicspath{{Figs/}}
\newtheorem{theorem}{Theorem}[section]
\newtheorem{proposition}[theorem]{Proposition}
\newtheorem{corollary}[theorem]{Corollary}
\newtheorem{lemma}[theorem]{Lemma}
\theoremstyle{definition}
\newtheorem{remark}[theorem]{Remark}
\newtheorem{definition}[theorem]{Definition}

\newcommand{\genset}{{\mathcal G}^*}
\newcommand{\gensetdyck}{{\mathcal G}}
\newcommand{\Geno}{Q^{*o}}
\newcommand{\catsuffix}{Q}
\newcommand{\cat}{M}
\newcommand{\catset}{{\mathcal M}}
\newcommand{\catsuffixset}{{\mathcal Q}}

\newcommand{\Cyl}{O}
\newcommand{\R}{\mathcal{R}}
\newcommand{\maj}{{\rm{maj}}}

\newcommand\touch[1]{\check{{#1}}}

\newcommand{\haut}{\operatorname{ht}}

\renewcommand{\H}{\operatorname{Heap}}

\newcommand{\Aaff}{\aff{A}}

\newcommand{\aff}[1]{\widetilde{#1}}

\newcommand\qbi[3]{{{#1}\atopwithdelims[]{#2}}_{#3}}
\newcommand\bi[2]{{{#1}\atopwithdelims(){#2}}}

\allowdisplaybreaks

\begin{document}

\title[Fully commutative involutions]{Combinatorics of fully commutative involutions in Classical Coxeter groups}

\author[Riccardo Biagioli]{Riccardo Biagioli}
\address{Institut Camille Jordan, Universit\'e Claude Bernard Lyon 1,
69622 Villeurbanne Cedex, France}
\email{biagioli@math.univ-lyon1.fr}
\urladdr{http://math.univ-lyon1.fr/{\textasciitilde}biagioli}

\author[Fr\'ed\'eric Jouhet]{Fr\'ed\'eric Jouhet}
\address{Institut Camille Jordan, Universit\'e Claude Bernard Lyon 1,
69622 Villeurbanne Cedex, France}
\email{jouhet@math.univ-lyon1.fr}
\urladdr{http://math.univ-lyon1.fr/{\textasciitilde}jouhet}

\author[Philippe Nadeau]{Philippe Nadeau}
\address{CNRS, Institut Camille Jordan, Universit\'e Claude Bernard Lyon 1,
69622 Villeurbanne Cedex, France}
\email{nadeau@math.univ-lyon1.fr}
\urladdr{http://math.univ-lyon1.fr/{\textasciitilde}nadeau}

\thanks{}

\date{\today}

\subjclass[2010]{}


\begin{abstract}
An element of a Coxeter group $W$ is  fully commutative if any two of its reduced decompositions are related by a series of transpositions of adjacent commuting generators. In the present work, we focus on fully commutative involutions, which are characterized in terms of Viennot's heaps. By encoding the latter by Dyck-type lattice walks, we enumerate fully commutative involutions according to their length, for all classical finite and affine Coxeter groups. In the finite cases, we also find explicit expressions for their generating functions with respect to the major index. Finally in affine type $A$, we connect our results to Fan--Green's cell structure of the corresponding Temperley--Lieb algebra.
\end{abstract}

\maketitle

%\tableofcontents

%%%%%%%%%%%%%%%%%%%%%%%%%%%%%%%%%%%%%%%%%%%
\section*{Introduction}
\label{sec:intro}

Let $W$ be a Coxeter group. An element $w \in W$ is said to be {\em fully commutative} (FC) if any reduced expression for $w$ can be obtained from any other one by transposing adjacent pairs of commuting generators. Fully commutative elements were extensively studied by Stembridge in a series of papers~\cite{St1, St2, St3} where, among others, he classified the Coxeter groups having a finite number of FC elements and enumerated them in each case. 
It is known that these elements index a basis for the associated (generalized) Temperley--Lieb algebra (\cite{Fan,Graham}). The growth function of such an algebra can then be obtained by computing the generating function, with respect to the length, of fully commutative elements in $W$. This function has been computed by Barcucci {\em et al}~\cite{BDPR} in type $A$, by Hanusa and Jones~\cite{HanJon} in type $\Aaff$, and by the present authors~\cite{BJN} in all finite and affine types. A striking fact is that, in each affine case, the corresponding growth sequence is ultimately periodic. 

\medskip

In the present work we focus on FC involutions for all classical Coxeter groups. As explained by Stembridge in~\cite{St3}, a FC element $w$ is an involution if and only if its commutation class $\mathcal{R}(w)$ is {\em palindromic}, meaning that it includes the mirror image of some (equivalently, all) of its members. We will reformulate this in terms of heaps and certain Dyck-type lattice walks which encode them. 

As a first consequence, we will be able to enumerate, in types $A$, $B$, and $D$, FC involutions according to the \emph{major index}. This was recently done by Barnabei {\em et al.}~\cite{BBES} for the symmetric group (type $A$), by using the $321$-avoiding characterization of such elements, the Robinson-Schensted correspondence, and a nice connection to integer partitions. Our approach in terms of heaps uses neither pattern-avoidance characterizations nor the Robinson-Schensted correspondence. However, it also yields a connection to integer partitions through our Dyck-type lattice walks, which, as will be explained, turn out to be in bijection with the ones used in~\cite{BBES}. The advantage of our point of view is that it naturally extends to types $B$ and $D$ for which major indices can be defined, whereas the use of Stembridge's pattern-avoiding characterizations~\cite[Theorems~5.1 and 10.1]{St3} seems hard to handle. In type $B$, our result can for instance be written as follows (see Proposition~\ref{prop:majorB}):
$$\sum_{w\in \bar{B}_{n}^{FC}}q^{\maj(w)}=\sum_{h=1}^nq^h\sum_{i=0}^{h-1}\qbi{h-1}{i}{q}+\qbi{n}{\lfloor n/2 \rfloor}{q},$$
where the square brackets are the so-called $q$-binomial coefficients, $\bar{B}_{n}^{FC}$ is the set of  FC involutions in $B_n$, and $\maj$ is the major index.

\medskip

As a second consequence of our characterizations of FC involutions, we will also give the $t$-weighted generating functions for all types of classical finite and affine Coxeter group. More precisely, if $\bar{W}^{FC}$ denotes the subset of FC involutions of $W$ and $\ell$ denotes the Coxeter length, we will define $\bar{W}^{FC}(t):=\sum_{w\in \bar{W}^{FC}}t^{\ell(w)}$ as the length generating function for FC involutions in $W$. We will use both our characterization in terms of heaps and the way these are encoded by Dyck-type lattice walks to compute  $\bar{W}^{FC}(t)$ when $W$ is finite or affine. In the affine case, we will also show that the corresponding growth sequences are ultimately periodic, with periods dividing the values recorded in the following table (see Propositions~\ref{prop:involutionsAaffines} and~\ref{prop:involutionsBCDaffines}):
$$\begin{array}{ l || c|c|c|c}
    \textsc{Affine Type} &\aff{A}_{n-1}\;(n\;\mbox{even})&\aff{C}_n&\aff{B}_{n+1}&\aff{D}_{n+2} \\ \hline
    \textsc{Periodicity} &n&2n+2&(2n+1)(2n+2)&2n+2\\
   \end{array}$$
   (if $n$ is odd, the number of fully commutative involutions in $\aff{A}_{n-1}$ is finite.)\\

Finally, as a third consequence of our approach, we will relate our previous characterization in affine type $\Aaff$ to Fan--Green's cells structure of  the associated Temperley--Lieb algebra described in~\cite{FanGreen_Affine}. More precisely, we will see how the use of heaps highlights and simplifies some of their results connected to FC involutions, such as~\cite[Theorem 3.5.1]{FanGreen_Affine}.

\medskip

This paper is organized as follows. In Section~\ref{sec:FCihw}, we recall definitions and basic results on Coxeter groups, fully commutative involutions, heaps and walks. In Section~\ref{sec:major}, we enumerate FC involutions with respect to the major index in classical finite types. Section~\ref{sec:length} is devoted to the characterization of FC involutions and their enumeration with respect to the Coxeter length in classical finite and affine types. Finally, in Section~\ref{sec:cells}, after recalling the cell structure on FC elements defined in~\cite{FanGreen_Affine} using the type $\Aaff$ Temperley--Lieb algebra, we show how the use of heaps makes the combinatorics of these cells more explicit.

%%%%%%%%%%%%%%%%%%%%%%%%%%%%%%%%%%%%%%%%%%%%%%%%%%%%%%%%%%%%%%%%%
\section{Fully commutative involutions, heaps and walks}\label{sec:FCihw}
%%%%%%%%%%%%%%%%%%%%%%%%%%%%%%%%%%%%%%%%%%%%%%%%%%%%%%%%%%%%%%%

%%%%%%%%%%%%%%%%%%%%%%%
\subsection{Coxeter groups, length and major index}\label{subsec:coxeter}
%%%%%%%%%%%%%%%%%%%%%%%

We refer to~\cite{Humphreys} for standard notations and terminology pertaining to general Coxeter groups. A {\em Coxeter system}  is a pair $(W,S)$ where $W$ is a group and $S\subset W$ is a finite set of generators for $W$ subject only to relations of the form $(st)^{m_{st}}=1$, where $m_{ss}=1$, and  $m_{st}=m_{ts}\geq 2$, for $s\neq t \in S$. If $st$ has infinite order we set $m_{st}=\infty$. These relations can be rewritten more explicitly as $s^2=1$ for all $s\in S$, and   
$$\underbrace{sts\cdots}_{m_{st}}=\underbrace{tst\cdots}_{m_{ts}},$$
where $m_{st} <\infty$. They are the so-called {\em braid relations}. When ${m_{st}}=2$, they are simply named {\em commutation relations}, $st=ts$. 

The Coxeter graph of $(W,S)$ will be denoted by $\Gamma$. The irreducible Coxeter systems corresponding to finite and affine Coxeter groups are completely classified (see~\cite{BjorBrenbook,Humphreys}), and the Coxeter graphs corresponding to the classical families are depicted in Figures~\ref{fig:finite_diagrams} and~\ref{fig:affinediagrams}. Here, and all along this paper, the indexing of the classical Coxeter graphs is slightly different from the more standard one  used in~\cite{St1}--\cite{St3}, and~\cite{BJN}, but it is more appropriate when one considers both the Coxeter length and the major index. 

\begin{figure}[!ht]
\begin{center}
\includegraphics[width=\textwidth]{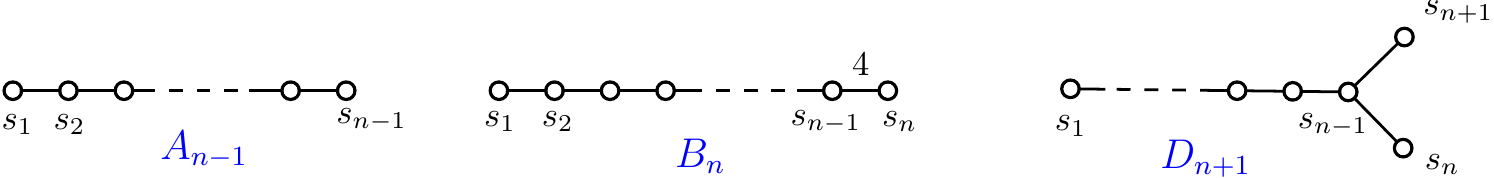}
\caption{\label{fig:finite_diagrams} Coxeter graphs for classical irreducible finite types.}
\end{center}
\end{figure}
\begin{figure}[!ht]
\begin{center}
 \includegraphics[width=\textwidth]{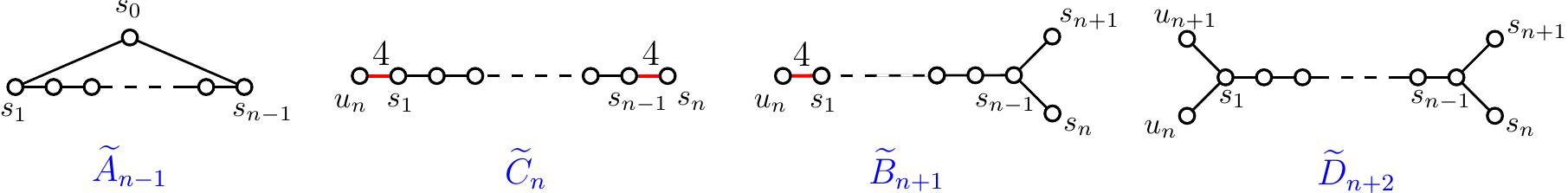}
 \caption{\label{fig:affinediagrams} Coxeter graphs for classical irreducible affine types.}
\end{center}
\end{figure}

For $w \in W$, the {\em length} of $w$, denoted by $\ell(w)$, is the minimum length $l$ of any expression $w=s_1\cdots s_l$ with $s_i\in S$. These expressions of length $\ell(w)$ are called \emph{reduced}, and we denote by $\mathcal{R}(w)$ the set of all reduced expressions of $w$. 

The {\em (right) descent set} of an element $w\in W$ is defined as follows:
$${\rm Des}_R(w)=\{ s \in S \mid \ell(ws)<\ell(w) \},$$ 
and its cardinality is usually called the {\em descent number} ${\rm des}(w)$. 

We define the {\em major index} of $w \in W$ as the sum of the labels of the descents of $w$; more precisely
$${\maj(w)}:=\sum_{s_i \in {\rm Des}_R(w)} i.$$
It is clear that the major index depends on the indexing of the generating set $S$ of $W$. In the case of the symmetric group $A_{n-1}$, the notion of major index is standard, and a very famous result states that the major index and the length are equidistributed over $A_{n-1}$ (see~\cite{MM, Foata}). For other Coxeter groups, several definitions exist. With our previous indexing of the Coxeter graphs,  the  major index defined above corresponds to the one already used by Reiner~\cite{Re}, and Steingr\'imsson~\cite{Stein} in the type $B$ case.

%%%%%%%%%%%%%%%%%%%%%%%
\subsection{Fully commutative elements and heaps}\label{subsec:heaps}
%%%%%%%%%%%%%%%%%%%%%%%

In this section, we recall the definition of fully commutative elements in Coxeter groups and its relation with the theory of heaps. We then explain how the  subfamily of alternating heaps corresponding to involutions can be encoded by various classes of Dyck-type lattice walks.
 
According to the well known {\em Matsumoto-Tits word property},  any expression in $\mathcal{R}(w)$ can be obtained from any other using only braid relations (see for instance~\cite[Section 3.3]{BjorBrenbook}). The notion of full commutativity is a strengthening of this property.

\begin{definition}
\label{defi:FC}
An element $w$ is {\em fully commutative} if any reduced expression for $w$ can be obtained from any other by using only commutation relations.
\end{definition}

We let $S^*$ be the free monoid generated by $S$. The equivalence classes of the congruence on $S^*$   generated by the commutation relations are usually called {\em commutation classes}. We recall from Stembridge in~\cite[Prop. 2.1]{St1} that an element $w\in W$ is FC if, and only if, all its reduced words avoid all factors of the form $sts\cdots$ with length $m_{st}\geq 3$. 

By definition the set $\mathcal{R}(w)$ forms a single commutation class; we will see that the concept of {\em heap} helps to capture the notion of full commutativity. We use the definition as given in~\cite[p. 20]{GreenBook} (see also~\cite[Definition 2.2]{KrattHeaps}), and we refer to \cite{BJN} for details of the relation with FC elements.

\begin{definition}[Heap]
\label{defi:heaps}
Let $\Gamma$ be a finite graph with vertex set $S$. A {\em heap} on $\Gamma$ (or {\em of type} $\Gamma$) is a finite poset $(H,\leq)$, together with a labeling map $\epsilon:H\to\Gamma$, which satisfies the following conditions:

\begin{enumerate}
\item For any vertex $s$, the subposet $H_s:=\epsilon^{-1}(\{s\})$ is totally ordered, and for any edge $\{s,t\}$, the subposet $H_{\{s,t\}}:=\epsilon^{-1}(\{s,t\})$ is totally ordered.
\item The ordering on $H$ is the transitive closure of the relations given by all chains $H_s$ and $H_{\{s,t\}}$, i.e. the smallest partial ordering containing these chains.
\end{enumerate}
\end{definition}

Two heaps on $\Gamma$ are {\em isomorphic } if there exists a poset isomorphism between them which preserves the labels.
The size $|H|$ of a heap $H$ is its cardinality. Given any subset $I\subset S$, we will note $H_{I}$ the subposet induced by all elements of $H$ with labels in $I$. In particular $H_{\{s\}}$ is the chain  $H_{s}=s^{(1)}<s^{(2)}<\cdots<s^{(k)}$ where $k=|H_s|$ is its cardinality. If $s,t$ are two labels such that $m_{st}\geq 3$, note that $H_{\{s,t\}}$ is also a chain.

 Fix a word $\mathbf{w}=s_{1}\cdots s_{l}$ in $S^*$. Define a partial ordering $\prec$ of the index set $\{1,\ldots, l\}$ as follows: set $i\prec j$ if $i<j$ and $\{s_{i},s_{j}\}$ is an edge of $\Gamma$, and extend by transitivity. We denote by $\H({\mathbf{w}})$ this poset together with $\epsilon:i\mapsto s_{a_i}$. It is easy to see that this indeed forms a heap as in Definition~\ref{defi:heaps}. %Also, words from the same commutation class are sent to isomorphic heaps.

\begin{proposition}[Viennot, \cite{ViennotHeaps}]
\label{prop:wordtoheap} 
 Let $\Gamma$ be a finite graph. The map $\mathbf{w}\to \H({\mathbf{w}})$ induces a bijection between $\Gamma$-commutation classes of words and isomorphism classes of finite heaps on  $\Gamma$ . 
\end{proposition}

As explained in~\cite{BJN}, if $w$ is a FC element and $\mathbf{w}$ is a reduced word for it, we can define $\H(w):=\H(\mathbf{w})$, and  heaps of this form are called {\em FC heaps}.
In Figure~\ref{fig:Heaps_Invol}, we consider the Dynkin diagram $B_7$, and we give two examples of words with the corresponding FC heaps. In the Hasse diagram of $\H(\mathbf{w})$, elements with the same labels are drawn in the same column. 

\begin{figure}[!ht]
\begin{center}
\includegraphics[width=12cm]{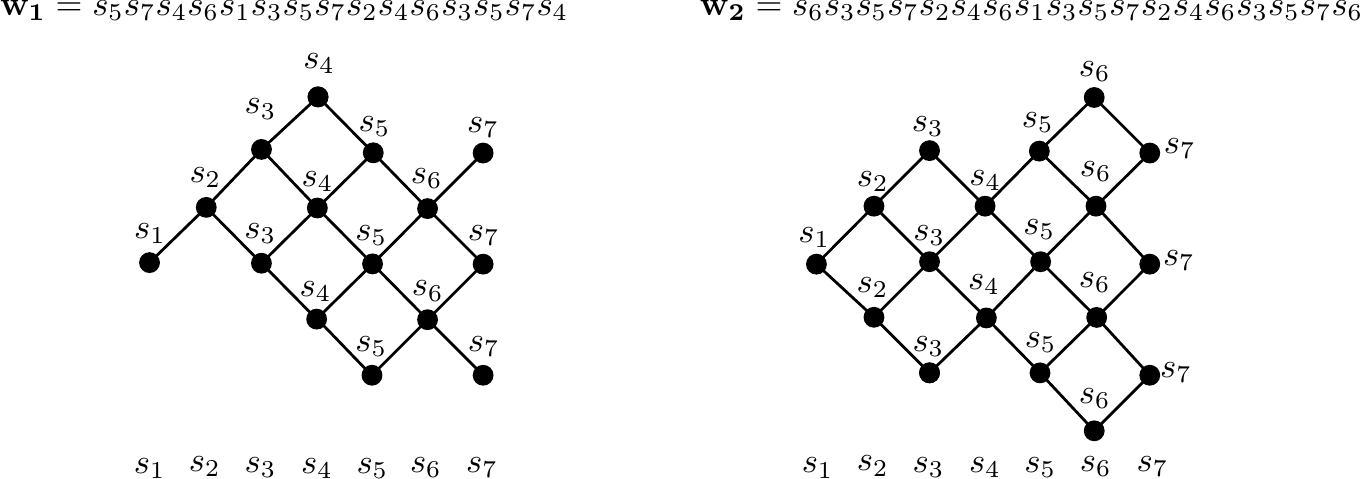}
\caption{Two heaps on $B_7$, the one on the right corresponding to a FC involution.}\label{fig:Heaps_Invol}  
\end{center}
\end{figure}

Let us point out a simple operation on heaps: if $(H,\leq,\epsilon)$ is a heap, then its {\em dual} is $(H,\geq,\epsilon)$, which is the heap with the inverse order, and where the labels are kept the same. We will say that a heap is \emph{self-dual} if it is isomorphic to its dual. Then we have the following result, which is essentially proved in~\cite{St3}.
\begin{lemma}\label{lemme:heapinvolution}
Let $W$ be a Coxeter group with corresponding Coxeter graph $\Gamma$. A fully commutative element $w\in W$ is an involution if and only if  $\H(w)$ is  self-dual.
\end{lemma}
\begin{proof}
A FC element $w \in W$ is an involution if and only if for a (equivalently, any) ${\bf w}=s_{i_1}\cdots s_{i_n} \in \R(w)$,  the reverse $\overleftarrow{\bf{w}}:= s_{i_n}\cdots s_{i_1}$ is also in $\R(w)$.  This is equivalent to $\H(w)=\H(\overleftarrow{\bf{w}})$, which is by definition the dual of $\H(w)$.   
\end{proof}

%%%%%%%%%%%%%%%%%%%%%%%
\subsection{Alternating heaps and self-dual right-peaks}\label{subsec:heaps_finite}
%%%%%%%%%%%%%%%%%%%%%%%
In this section we recall the classification of the FC involutions in the classical finite Coxeter groups essentially given in~\cite{St3}. We use a description in terms of heaps derived from~\cite[\S 2.5, \S 4.4]{BJN}, together with Lemma~\ref{lemme:heapinvolution}.
\smallskip

First, we need to recall a few definitions. We fix $m_{v_0v_1},m_{v_1v_2},\ldots, m_{v_{n-1}v_{n}}$ in the set $\{3,4,\ldots\}\cup\{\infty\}$ and we consider the Coxeter system $(W,S)$ corresponding to the {\em linear} Coxeter graph $\Gamma_n=\Gamma_n((m_{v_iv_{i+1}})_i)$ of Figure~\ref{fig:linear_diagram}.

\begin{figure}[h!]
\begin{center}
\includegraphics{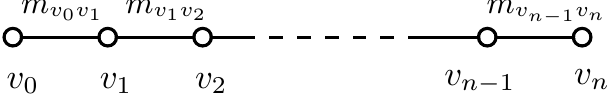}
\caption{\label{fig:linear_diagram} The linear Coxeter graph $\Gamma_n$.}
\end{center}
\end{figure}

Note that this linear diagram contains as special cases the Coxeter diagrams of types $A_{n-1}$, $B_n$ and $\tilde{C}_{n}$.

\begin{definition}
\label{defi:alternating}
Consider the linear graph $\Gamma_n$, and a heap $H$ on $\Gamma_n$. We say that $H$ is {\em alternating} if for each edge $\{v_i,v_{i+1}\}$ of $\Gamma_n$, the chain $H_{\{v_i,v_{i+1}\}}$ has alternating labels $v_i$ and $v_{i+1}$.
\end{definition}
Next, we will need the following families of heaps, which are specific to the types $B$ and $D$.
\begin{definition}
A {\em self-dual right-peak of type $B_n$} is a heap such that there exists $j \in\{1,\dots,n-1\}$ satisfying:
\begin{enumerate}
\item[(a)]  $H_{\{s_j,\ldots, s_n\}}=\H(s_j\cdots s_{n-1}s_n s_{n-1} \cdots s_j)$;
\item[(b)] $H_{\{s_{j-1}, s_{j}\}}=s_j s_j$ or $s_{j-1}s_js_js_{j-1}$ for $j>1$, and $s_1s_1$ for $j=1$; 
\item[(c)]  $H_{\{{s}_{1},\ldots,\dot{s}_{j} \}}$ is self-dual alternating, where $\dot{s}_j$ means that one $s_j$ is deleted.
\end{enumerate}
\end{definition}

\begin{figure}[h!]
\begin{center}
\includegraphics[height=3cm]{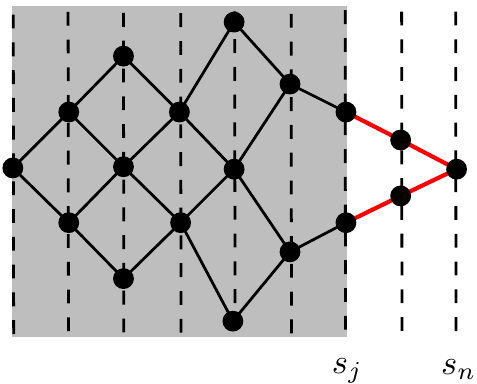}
\caption{\label{fig:RP_typeB} A self-dual right-peak of type $B_n$.}
\end{center}
\end{figure}

The Coxeter diagram of type $D_{n+1}$ is not linear. We can nevertheless define self-dual alternating heaps and self-dual right-peaks based on the previous definitions in type $B_n$. More precisely, {\em self-dual alternating heaps of type} $D_{n+1}$ are obtained from self-dual alternating heaps of type $B_n$ having either zero or an odd number of elements labeled $s_n$, by replacing these by $s_n$ or $s_{n+1}$ alternatively, or by $s_ns_{n+1}$  if there was originally exactly one label $s_n$, cf. Figure~\ref{fig:alt_typeD}.

\begin{figure}[h!]
\begin{center}
\includegraphics[height=3cm]{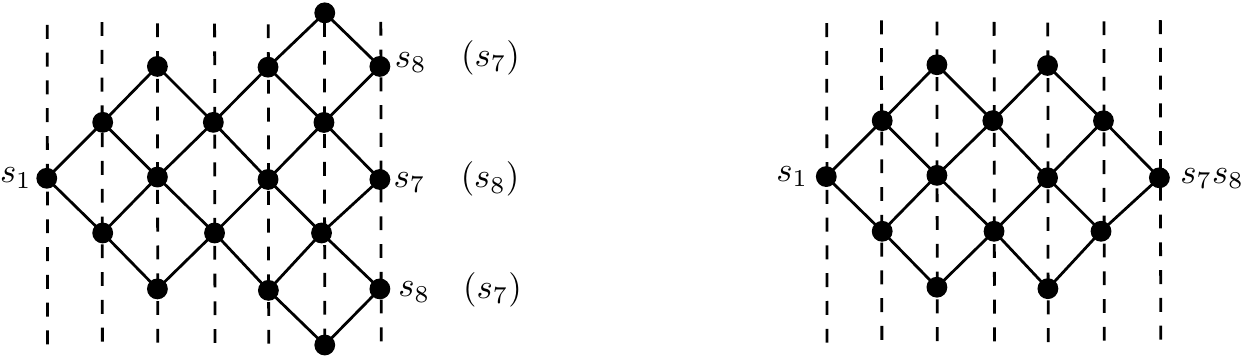}
\caption{\label{fig:alt_typeD} Two self-dual alternating heaps of type $D_8$.}
\end{center}
\end{figure}

{\em Self-dual right peaks of type} $D_{n+1}$ are heaps satisfying conditions $(b)$ and $(c)$ above, and such that there exists $j \in\{1,\dots,n-1\}$ with:
\begin{enumerate}
\item[{\em (a')}] $H_{\{s_j,\ldots, s_{n+1}\}}=\H(s_j\cdots s_{n-1}s_ns_{n+1} s_{n-1} \cdots s_j)$.
\end{enumerate}

\begin{remark}
\label{rem:familyPeaks}
In the above families of right-peaks the index $j$ is {\em uniquely determined}; this will be particularly useful for enumerating purposes.
\end{remark}
From  \cite[Theorem 3.10]{BJN} and Lemma~\ref{lemme:heapinvolution}, we have the following result.
\begin{proposition}[Classification of FC involutions in classical types]\label{prop:typeclassiques}
A FC element $w \in A_{n-1}$ is an involution if and only if $\H(w)$ is a self-dual alternating heap. Moreover, a FC element $w \in B_n$ (\emph{resp.} $w \in D_{n+1}$)   is an involution if and only if $\H(w)$ is either a self-dual alternating heap of type $B_n$ (\emph{resp.}  $D_{n+1}$) or a self-dual right-peak of type $B_n$ (\emph{resp.} $D_{n+1}$). 
\end{proposition}

%%%%%%%%%%%%%%%%%%%%%%%
\subsection{Motzkin type lattice walks}\label{subsec:walks}
%%%%%%%%%%%%%%%%%%%%%%%

Next, we define the lattice walks we will use in the sequel to compute generating functions for FC involutions.

\begin{definition}[Walks]
\label{defi:walks}
A {\em walk} of length $n$ is a sequence $P=(P_0,P_1,\ldots,P_n)$ of points in $\mathbb{N}^2$ with its $n$ steps in the set $\{(1,1),(1,-1),(1,0)\}$, such that $P_0$ has abscissa $0$ and all horizontal steps $(1,0)$ can only occur between points on the $x$-axis.
\end{definition}

The set of all walks of length $n$ will be denoted by $\genset_n$. The subset of walks starting at $P_0=(0,0)$ will be denoted by $\catsuffixset^*_n$, and the subset of $\catsuffixset^*_n$  ending at $P_n=(n,0)$ will be denoted by  $\catset^*_n$. To each family $\mathcal{F}^*_n \subseteq \genset_n$ corresponds the subfamily ${\mathcal{F}_n}$ consisting of those walks with no horizontal step, and $\touch{\mathcal{F}_n}\subseteq\mathcal{F}_n$ consisting of the ones which hit moreover the $x$-axis at some point. %\smallskip

  The \emph{total height} $\haut$ of a walk is the sum of the heights of its points: if $P_i=(i,h_i)$ then $\haut(P)=\sum_{i=0}^n h_i$. To each family $\mathcal{F}^*_n \subseteq \genset_n $ we associate the series $F_n^*(t)=\sum_{P\in \mathcal{F}^*_n}t^{\haut(P)}$, and we define the generating functions in the variable $x$ by 
  
\[ F^*(x)=\sum_{n\geq 0} F_n^*(t) x^n, \quad F(x)=\sum_{n\geq 0} F_n(t) x^n \quad {\rm and} \quad \touch{F}(x)=\sum_{n\geq 0}\touch{ F}_n(t) x^n.\]

Notice that $\cat(x)$ (\emph{resp.} $\catsuffix(x)$)  counts  Dyck paths (\emph{resp.} prefixes of Dyck paths), taking into account the height. Moreover, the generating function  $\cat^*(x)$ can easily be related to  $\cat(x)$: a path in $\catset^*_n$ can be seen as a concatenation of Dyck paths and horizontal steps which alternate, remembering that the set of Dyck paths contains the empty one. Therefore one can write:
\begin{equation}\label{lienDyck}
\cat^*(x)=\frac{\cat(x)}{1-x\cat(x)}.
\end{equation}

Recall that, setting $t=1$, the series $\cat(x)$ simply becomes the generating function for the famous Catalan numbers, and is therefore equal to the generating function for $321$-avoiding permutations, or equivalently FC elements of type $A$ (see~\cite{BJN,BJS,St3}).

Moreover, we have the recursive relation
\begin{equation}\label{eqfunccat}
\cat(x)=1+tx^2\cat(x)\cat(tx),
\end{equation}
as can be seen by considering the first return to the $x$-axis, as follows:

\begin{center}
\includegraphics[width=11cm]{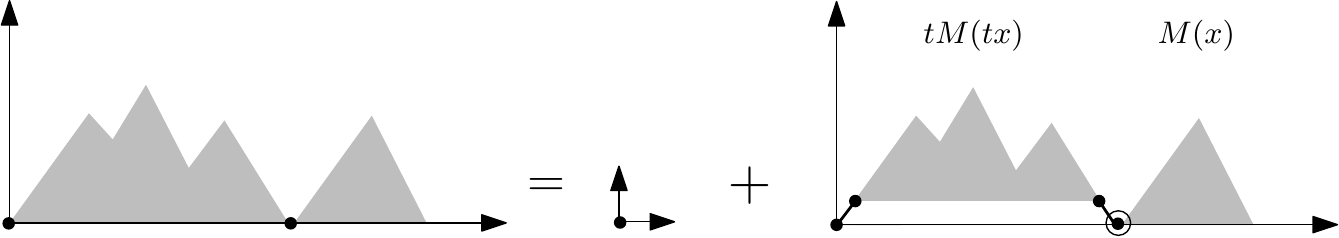}
\end{center}

\smallskip
We also point out that for all the subfamilies $\mathcal{F}^*_n$ of $\genset_n$ that will naturally appear later in our enumerations, there will be some equations of the form~\eqref{lienDyck} relating $F^*(x)$, $F(x)$ and $M(x)$, and recursive relations like~\eqref{eqfunccat} for $F(x)$.

%%%%%%%%%%%%%%%%%%%%%%%
\subsection{Bijective encoding of alternating heaps}\label{subsec:walks}
%%%%%%%%%%%%%%%%%%%%%%%

We now define a bijective encoding of self-dual alternating heaps by walks, which will be especially handy to compute generating functions in the next sections. It is obtained by understanding how the encoding of heaps defined in~\cite{BJN} restricts to self-dual heaps.

\begin{proposition}
\label{proposition:walk_encoding}
Let $H$ be a self-dual alternating heap of type $\Gamma_n$. To each vertex  $v_i$ of $\Gamma_n$ we associate the point $P_i=(i,|H_{v_i}|)$.  We define $\varphi(H)$ as the walk $(P_0,P_1,\ldots,P_n)$. Then the map $H\mapsto \varphi(H)$ is a bijection between self-dual alternating heaps of type $\Gamma_n$ and $\genset_n$. The size $|H|$ of the heap is the total height of $\varphi(H)$.
\end{proposition}

\begin{proof}
For any alternating heap $H$ and any $i\in\{0,\dots,n-1\}$, we have $-1 \leq |H_{v_i}| - |H_{v_{i+1}}| \leq 1$. Moreover if $H$ is self-dual and $|H_{v_i}|\neq 0$, then $|H_{v_i}| \neq |H_{v_{i+1}}|$. Therefore the map is well defined. Fix $(P_0,P_1,\ldots,P_n) \in \mathcal{G}_n^*$. If the step $P_i=(i,h_i) \rightarrow P_{i+1}=(i+1,h_{i+1})$ is equal to $(1,1)$ (\emph{resp.} $(1,-1)$), then we define a convex chain $\mathcal{C}_i$ of length $2h_{i}$ (\emph{resp.} $2h_{i+1}$) as $(v_{i+1},v_{i},\ldots, v_{i+1})$  (\emph{resp.} $(v_i,v_{i+1},\ldots, v_i)$).
If the step $P_i\rightarrow P_{i+1}$ is $(1,0)$, then $h_i=h_{i+1}=0$, and we let $\mathcal{C}_i$ be the empty chain. Next we define $H$ as the transitive closure of the chains $\mathcal{C}_0, \ldots,\mathcal{C}_{n-1}$. It is acyclic since $\Gamma_n$ is linear, so we get a heap $H$ which is uniquely defined, alternating, and satisfies $\varphi(H)=(P_0,P_1,\ldots,P_n)$. Since $h_i=|H_{v_i}|$ the result follows.
\end{proof}

\begin{figure}[h!]
\begin{center}
\includegraphics[width=11cm]{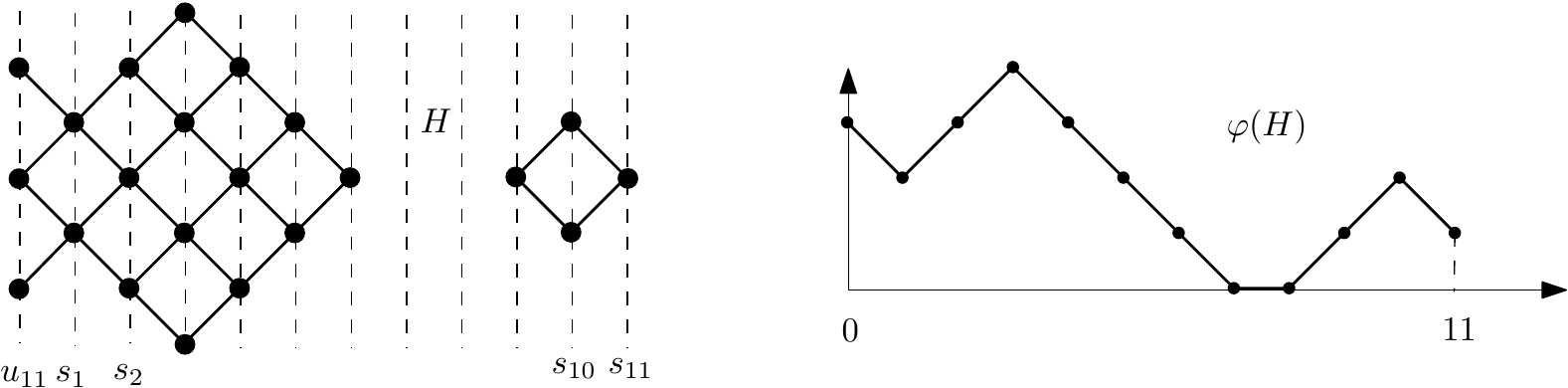}
\caption{\label{fig:C_invol_path} The heap of a FC involution in $\widetilde{C}_{11}$ and its associated walk.}
\end{center}
\end{figure}

\begin{remark}\label{rk:ajouterunpas}
In the type $A_{n-1}$ case, if we consider alternating FC heaps, since they have either zero or one occurrence of the labels $s_1$ and $s_{n-1}$, we can add an initial and a final step to the paths obtained through Proposition~\ref{proposition:walk_encoding}, to get a bijection with paths in $\catset^*_n$. Similarly, since alternating FC heaps on $B_n$ are in bijection with paths in $\genset_{n-1}$ starting at height zero or one, we can add an extra initial step, to obtain a bijection with paths in $\catsuffixset^*_n$.
\end{remark}

%%%%%%%%%%%%%%%%%%%%%%%%%%%%
\section{Enumeration with respect to the major index}\label{sec:major}
%%%%%%%%%%%%%%%%%%%%%%%%%%%%

Recall the following enumerative expressions for the number of FC involutions in classical finite types, which were first proved in~\cite[Section 4]{St3}.

\begin{proposition}\label{coro:enumeration}
We have for $n\geq1$
\begin{eqnarray}
|\bar{A}_{n-1}^{FC}|&=&\bi{n}{\lfloor n/2 \rfloor}\label{cardA},\\
|\bar{B}_{n}^{FC}|&=&2^n+\bi{n}{\lfloor n/2 \rfloor}-1\label{cardB},\\
|\bar{D}_{n+1}^{FC}|&=&\left\{\begin{matrix} \displaystyle 2^n+\bi{n+1}{n/2}-1\;\;\;\mbox{if $n$ even},\\ 
\\
\displaystyle2^n+\frac{3}{2}\bi{n+1}{ (n+1)/2 }-1\;\;\;\mbox{if $n$ odd.} \end{matrix}\right.\label{cardD}
\end{eqnarray}
\end{proposition}

The striking observation is that by considering the major index defined in Section~\ref{sec:FCihw} we can obtain nice $q$-analogues for formulas \eqref{cardA}--\eqref{cardD}. For type $A$ this analogue was recently computed by Barnabei \emph{et al.} in~\cite{BBES}, by using the $321$-avoiding characterization of FC involutions, the Robinson--Schensted algorithm, and a connection with integer partitions. A non-bijective proof of this result, using the principal specialization of Schur functions, has also been given by Stanley,  and a third one has been found by Dahlberg and Sagan (see~\cite[\S6]{BBES}).

In this section, we show that our different approach in terms of heaps will allow us to derive this result and extend it to types $B$ and $D$. In type $A$, we will shortly explain in Proposition~\ref{remark:rsk} below that our approach and the one of Barnabei \emph{et al.} are equivalent, and yield the same connection to integer partitions. In types $B$ and $D$, we will still use our characterization in terms of heaps and the idea from~\cite{BBES} to connect walks to integer partitions; here the approach through pattern avoidance in signed permutations seems harder to use.
\smallskip

Before giving our result, we recall a classical property relating $q$-binomial coefficients and integer partitions. A partition of the integer $N$ is a nonincreasing sequence $\lambda=(\lambda_1\geq\dots\geq\lambda_l>0)$ of positive integers whose sum is equal to $N$, which we denote by $|\lambda|$ and call the weight of $\lambda$. The integer $l$ is the length of $\lambda$, while the $\lambda_i$'s are called the parts of $\lambda$. It will be convenient to use {\em Frobenius notations} for partitions. Let $j$ be the {\em rank} of a partition $\lambda$, namely the largest integer $i$ such that $\lambda_i\geq i$: it is also the size of the \emph{Durfee square} of $\lambda$. For $1\leq i \leq j$, define $\alpha_i=\lambda_i-i$ and $\beta_i=\lambda'_i-i$. The Frobenius notation for $\lambda$ is the array
$$\left(\begin{matrix} \displaystyle \alpha_1 & \alpha_2 & \dots & \alpha_j\\ 
\displaystyle \beta_1& \beta_2 & \dots & \beta_j \end{matrix}\right).$$

We have the following classical generating function (see for instance~\cite{Andrews}).
\begin{lemma}\label{lem:partitions}
The generating function in the variable $q$, according to the  weight, of integer partitions with largest part smaller or equal to $k$ and length smaller or equal to $n-k$, is given by the $q$-binomial coefficient
$$\qbi{n}{k}{q}:=\frac{(1-q)\cdots(1-q^n)}{(1-q)\cdots(1-q^k)(1-q)\cdots(1-q^{n-k})}.$$
\end{lemma}

\subsection{Type $A$}
We  prove the following result, which can be found in~\cite{BBES}.
\begin{proposition}\label{prop:majorA}
For any positive integer $n$, we have
\begin{equation}\label{cmajA}
\sum_{w\in \bar{A}_{n-1}^{FC}}q^{\maj(w)}=\qbi{n}{\lfloor n/2 \rfloor}{q}.
\end{equation}
\end{proposition}

\begin{proof}
Recall from Proposition~\ref{prop:typeclassiques} that each FC involution $w \in \bar{A}^{FC}_{n-1}$ corresponds bijectively to a self-dual alternating heap $H=\H(w)$, which is itself in one-to-one correspondence with a walk in $\catset^*_n$, thanks to Proposition~\ref{proposition:walk_encoding} and Remark~\ref{rk:ajouterunpas}. Note that 
$i \in {\rm Des}_R(w)$ if and only if there exists in the poset $H$ a maximum element labeled $s_i$. In the path encoding this corresponds to a {\em peak} (i.e. an ascending step followed by a descending one) at coordinates
$(i,|H_{s_i}|)$. Therefore the position $i$ of such a descent is equal to the number of steps from the origin to this peak. 

Fix a walk $\gamma \in \catset^*_n$ and let $k$ be its number of horizontal steps. Note that  $k$ and $n$ have the same parity. Replace  the first $\lfloor k/2 \rfloor$ (\emph{resp.} the last $\lceil k/2 \rceil$)  horizontal steps by  descending (\emph{resp.} ascending) ones. Next replace  all descending (\emph{resp.} ascending) steps by horizontal (\emph{resp.} vertical) ones, yielding $\bar{\gamma}$ in $\Pi_n$,  the set of walks with $n$ horizontal or vertical steps, starting at the origin and ending at $(\lfloor n/2 \rfloor,\lceil n/2 \rceil)$. These operations are invertible, and an example of the bijection $\gamma\to\bar{\gamma}$ is provided in Figure~\ref{fig:bijection_typeA}.

\begin{figure}[h!]
\begin{center}
\includegraphics[width=12cm]{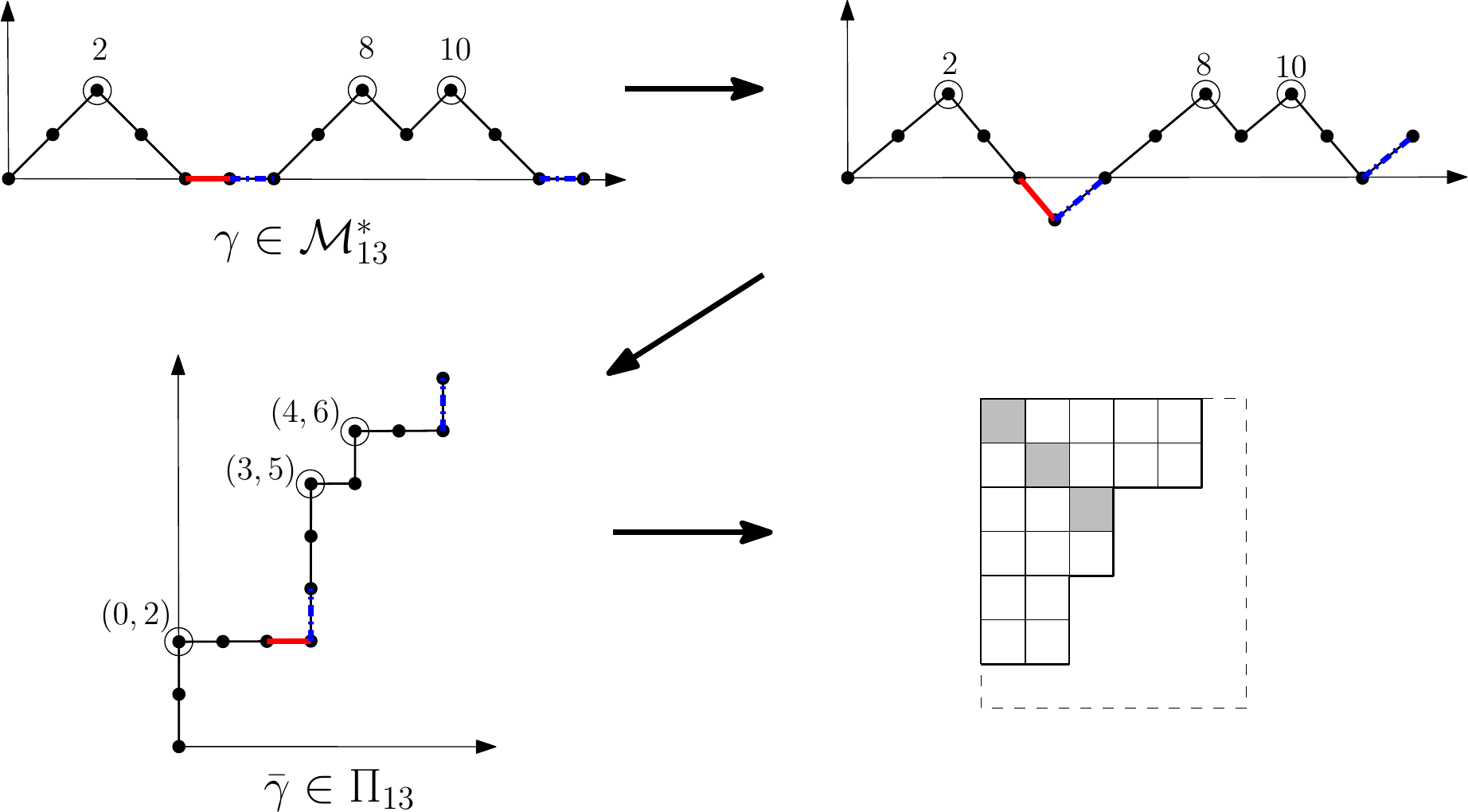}
\end{center}
\caption{An example of the walk-to-partition bijection.\label{fig:bijection_typeA}}
\end{figure}

It is not difficult to see that descents of $\gamma$ coincide in $\bar{\gamma}$ with {\em corner points} (formed by a vertical step followed by a horizontal one). Moreover the position of a descent in $\gamma$ is given by the sum of the coordinates of the corresponding corner point in $\bar{\gamma}$. 

Denote  by $P_1(a_1,b_1),\dots,P_j(a_j,b_j)$, with  $\lfloor n/2 \rfloor> a_1>\dots>a_j\geq 0$ and $\lceil n/2 \rceil \geq b_1>\dots>b_j> 0$ the corner points of $\bar{\gamma}$. Summarizing, we have
\begin{equation}\label{eq:sum_maj}
\sum_{w\in \bar{A}_{n-1}^{FC}} q^{\maj(w)}=\sum_{\bar{\gamma} \in \Pi_n} q^{a_1+b_1+\cdots + a_j+b_j}.
\end{equation}

Since any walk $\bar{\gamma}\in \Pi_n$ is uniquely determined by its corner points, it can be associated with an integer partition $\lambda$  defined in Frobenius notation as
$$\lambda_{\bar{\gamma}}=\left(\begin{matrix} \displaystyle a_1&a_2& \dots & a_j\\ 
\displaystyle b_1-1&b_2-1&\dots&b_j-1 \end{matrix}\right).$$ This is a bijection: for a proof see ~\cite[Theorem 3.4]{BBES} to which we refer for more detail. As any such partition has  length~$\leq \lceil n/2 \rceil$ and greatest part~$\leq  \lfloor n/2 \rfloor$, we obtain~\eqref{cmajA} through \eqref{eq:sum_maj} and Lemma~\ref{lem:partitions}.
\end{proof}

By comparing our resulting bijection with that of~\cite{BBES}, we observed a striking coincidence.
Starting from a FC involution $w \in A_{n-1}$, we considered the following sequence of bijective transformations:
$$w \longrightarrow H\longrightarrow \gamma \in  \catset^*_n,$$
where $H=\H(w)$ and $\gamma=\varphi(H)$ (see Remark~\ref{rk:ajouterunpas}).
Barnabei \emph{et al.}~\cite{BBES} consider a different sequence of transformations, namely:
$$w \stackrel{RS}{\longrightarrow} T\longrightarrow \beta \in \catsuffixset_n,$$
where $T$ is the two rows standard tableau obtained by the Robinson--Schensted correspondence, and $\beta$ is the classical associated Dyck-path prefix. The following proposition shows that the two constructions essentially coincide.
\begin{proposition}\label{remark:rsk}
The walk $\beta$ is obtained from $\gamma$ by replacing in the latter all horizontal steps by ascending ones. 
\end{proposition}
\begin{proof}

 If $w(i)=i$, then $w(j)<i$ for all $j<i$ (equivalently, $w(j)>i$ for all $j>i$): otherwise $w(j)>i>j$ would be an occurrence of the pattern $321$. This shows that $H$ splits naturally into two heaps corresponding to elements smaller, \emph{resp.} larger than $i$. Moreover $H$ contains elements labeled neither by $s_{i-1}$ nor by $s_i$, which means the $i$th step of $\gamma$ being a horizontal step. In the RS construction, we also easily see that such a fixed point $i$ will correspond to an up step in $\beta$, where $\beta$ leaves at abscissa $i$ a certain height for the last time. 

Thus the statement of the proposition is proved if we can show that $\beta=\gamma$ whenever $w$ has no fixed points, which we now assume. In this case $\beta$ and $\gamma$ are Dyck paths in $\catset_n$, and their construction is illustrated in Figure~\ref{fig:RSK_et_Heaps}. 

\begin{figure}[!ht]
\begin{center}
\includegraphics[width=0.8\textwidth]{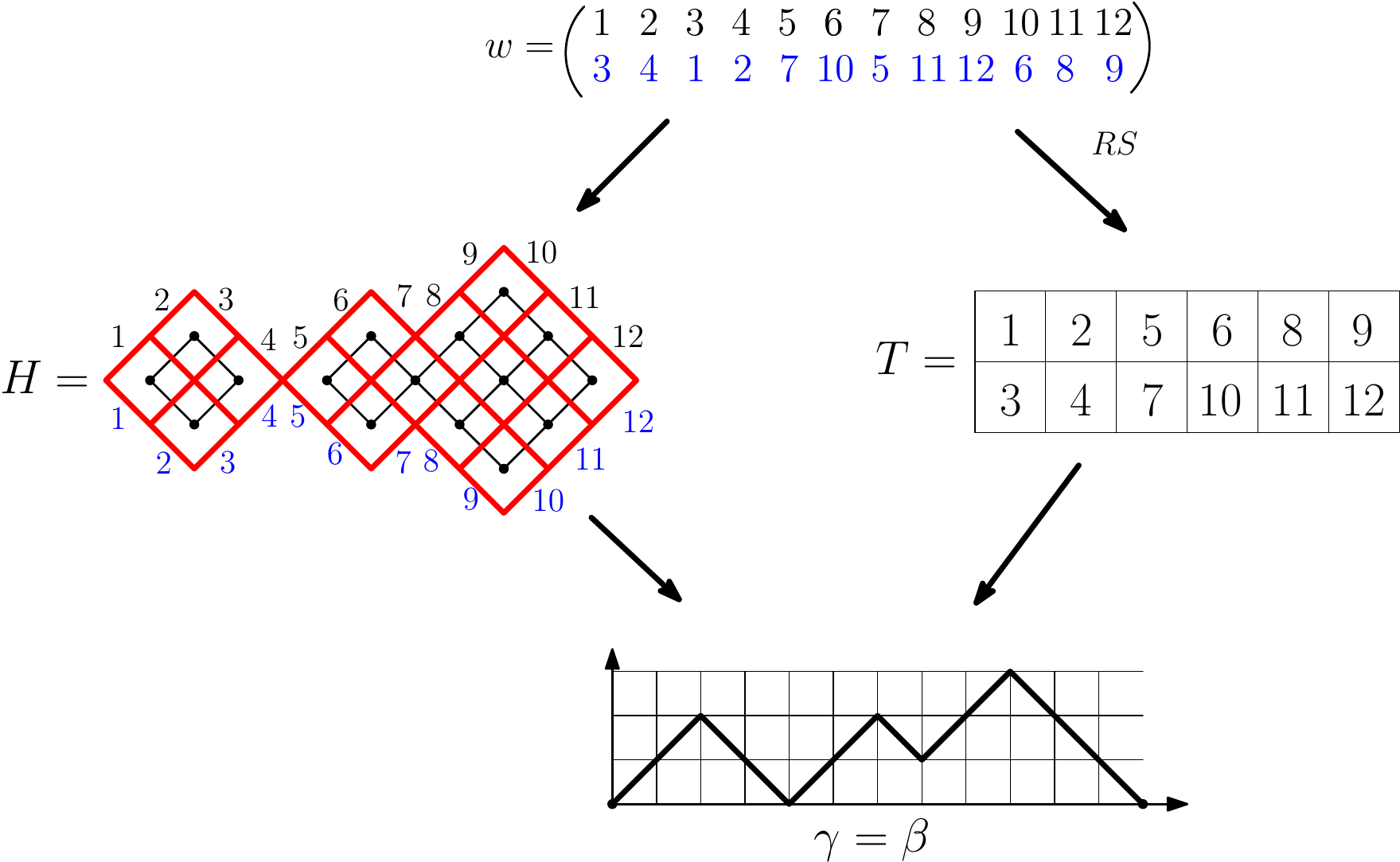}
\end{center}
\caption{Two constructions of the same correspondence.\label{fig:RSK_et_Heaps}}
\end{figure}

Note that we superimposed a cell on each vertex from the heap $H$. This is useful to read off the involution directly, see for instance~\cite{ViennotHeaps}. More precisely, pick any labeled edge, say $i$, on the top side, and follow the corresponding ``column'' of cells until the bottom side is reached: the label that can be read off is $w(i)$. 
By the construction of $\gamma$ from $w$, it is then clear that its $i$th step is an up step if and only if $w(i)>i$. 

On the other hand, the $i$th step in $\beta$ is an up step if and only if it appears in the top row of the tableau $T$ obtained from $w$. Now one can show by induction that, in the RS insertion algorithm, an entry $j<i$ will bump $i$ from the first row if and only if $j=w(i)$. Putting things together, the $i$th step in $\beta$ is an up step if and only if $w(i)>i$. Comparing with the previous paragraph, this shows that $\beta$ and $\gamma$ coincide, and concludes the proof.
\end{proof}

\subsection{Type $B$}
In type $B$, we have the following generalisation.
\begin{proposition}\label{prop:majorB}
For any positive integer $n$, we have
\begin{equation}\label{cmajB}
\sum_{w\in \bar{B}_{n}^{FC}}q^{\maj(w)}=\sum_{h=1}^nq^h\sum_{i=0}^{h-1}\qbi{h-1}{i}{q}+\qbi{n}{\lfloor n/2 \rfloor}{q}.
\end{equation}
\end{proposition}

\begin{proof}
 By Proposition~\ref{prop:typeclassiques}, we have to inspect separately self-dual alternating heaps and right-peaks of type $B_n$. By definition right-peaks can be split for a $j\in\{1,\dots,n-1\}$ into two heaps $H_{\{s_1,\ldots, \dot{s}_{j}\}}$ and $H_{\{s_j,\ldots, s_n\}}$, and this decomposition is unique by  Remark~\ref{rem:familyPeaks}. These two parts can be bijectively reassembled by associating to the first one a walk in $\catsuffixset^*_{j}$ ending at height one, and to the second one a walk starting with a $(1,-1)$-step, followed by $n-j-1$ horizontal steps, and then gluing these two walks together. This yields a walk in $\catset^*_n$ with less than $n$ horizontal steps. For instance, the walk associated to the right-peak in Figure~\ref{fig:RP_typeB} is depicted below.

\begin{center}
\includegraphics[width=4cm]{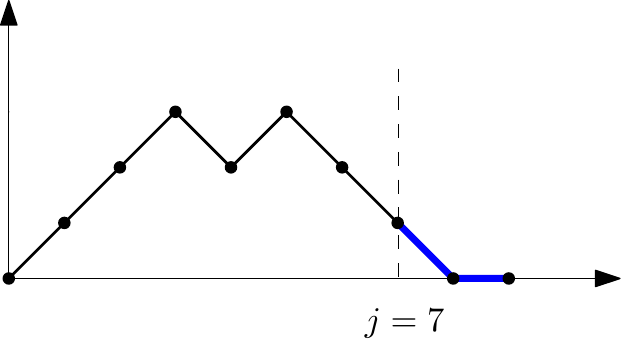}%\label{fig:chemin_peak_typeB}
\end{center}

We remark that if $w\in\bar{B}_n^{FC}$ corresponds to a right-peak, then $s_n \not\in {\rm Des_R(w)}$, so by applying the same arguments as in type $A_{n-1}$, we find that the generating function for those elements is equal to 
\begin{equation}\label{eq:LPmaj}
\qbi{n}{\lfloor n/2 \rfloor}{q}-1. 
\end{equation}

The remaining heaps are the self-dual alternating ones, which are by Proposition~\ref{proposition:walk_encoding} in one-to-one correspondence with walks in $\catsuffixset^*_n$. 

First, replace in these walks all $(1,0)$ steps by $(1,-1)$ steps. Secondly, replace as in type $A_{n-1}$ all descending (\emph{resp.} ascending) steps by horizontal (\emph{resp.} vertical) ones. We obtain this time a bijection with length $n$ walks with horizontal or vertical steps, starting at the origin but ending at any height, whose set is denoted by $\Pi'_n$, as can be seen on the illustration below.

\begin{center}
\includegraphics[width=9cm]{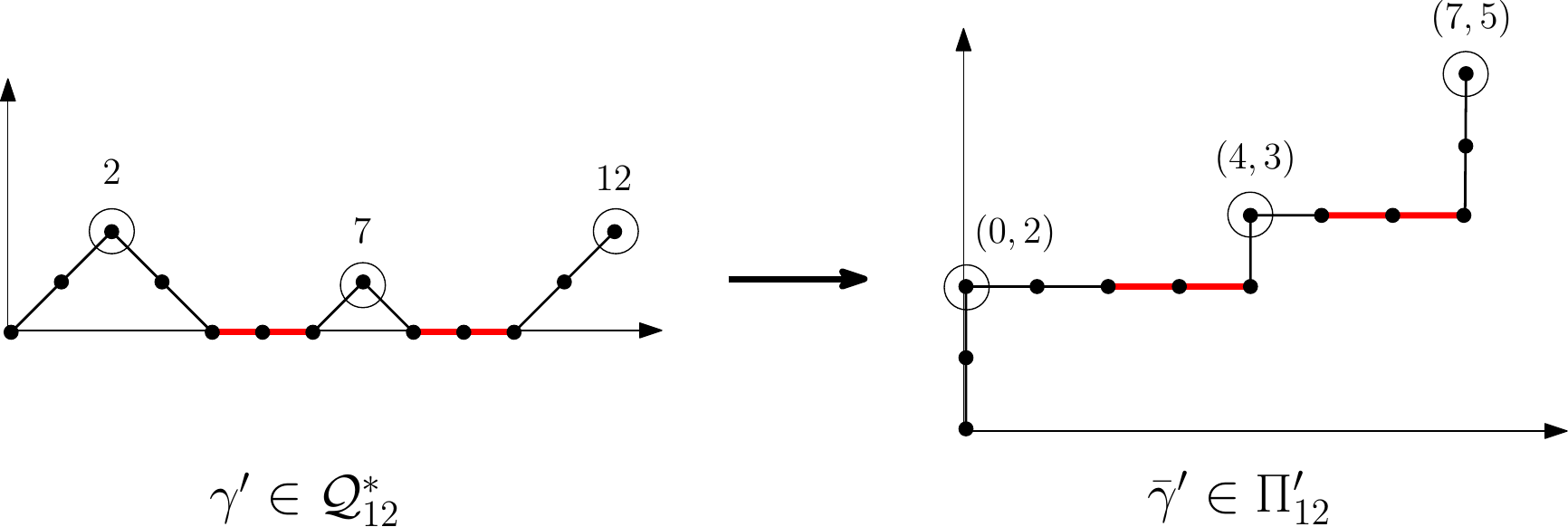}\label{fig:chemin_alt_typeB}
\end{center}

As before, if we fix $w\in\bar{B}_n^{FC}$, the descents of $w$ correspond in $\bar{\gamma}'\in \Pi'_n$ to corner points, together with  the last point if  it ends a vertical step.  Denote all these points by $P_1(a_1,b_1),\dots,P_j(a_j,b_j)$; they satisfy  $a_1>\dots>a_j\geq 0$ and $b_1>\dots>b_j>0$, with the additional condition $a_1+b_1\leq n$, and again $\maj(w)$ is equal to $a_1+b_1+\dots+a_j+b_j$. As in type $A$, the corner points  define a unique partition $\lambda_{\bar{\gamma}'}$ in Frobenius notation, such that  $a_1+b_1\leq n$. 

Summarizing,  alternating FC involutions in $B_n$ are in bijection with such integers partitions $\lambda_{\bar{\gamma}'}$ satisfying $a_1+b_1\leq n$. If  $\lambda_{\bar{\gamma}'}$ is not empty, note that $a_1=\lambda_1$ and $b_1=l-1$, where $l$ is its length. Set $h:=a_1+b_1$ and $i:=b_1$, therefore we have $1\leq h\leq n$ and $0\leq i\leq h-1$. Remove  from $\lambda_{\bar{\gamma}'}$ its first part $\lambda_1$ and decrease by one all the remaining parts, therefore giving a partition $\mu$ of weight $|\lambda_{\bar{\gamma}'}|-h$. Moreover,  $\mu$ satisfies the conditions of Lemma~\ref{lem:partitions} with $n$ replaced by $h-1$ and $k$ by $i$, therefore giving the generating function
$$1+\sum_{h=1}^nq^h\sum_{i=0}^{h-1}\qbi{h-1}{i}{q},$$
which concludes the proof.
\end{proof}

Thanks to the previous proof, we can refine as follows  this result, as was done in~\cite{BBES} for type $A$.
\begin{corollary}
For any integer  $1\leq k\leq n$, we have
$$\sum_{\stackrel{w \in \bar{B}_n^{FC}}{{\rm des(w)=k}}} q^{\maj(w)}=q^{k^2}\sum_{h= 0}^{n-2k-1} q^h 
\sum_{i=0}^{h}  \qbi{i+k-1}{k-1}{q} \qbi{h+k-1-i}{k-1}{q}.$$ 
Moreover, when $k=0$, this generating function is equal to $1$. 
\end{corollary}
\begin{proof}
We analyse more precisely the structure of the partition $\lambda_{\bar{\gamma}'}$ from the previous proof when the  FC involution has $k$ descents. Obviously, if $k=0$, then this partition is empty, which yields the result. Otherwise $k\geq1$ is the row length in the Frobenius notation of $\lambda_{\bar{\gamma}'}$, so it  corresponds to the size of its Durfee square. We can define as before the integers $h$ and $i$, satisfying this time $2k-1\leq h\leq n$ and $k-1\leq i\leq h-k$, respectively. Moreover, the partition $\mu$ obtained from $\lambda_{\bar{\gamma}'}$ has a Durfee square of size $k-1$, so the desired generating function is equal to
$$\sum_{h= 2k-1}^{n} q^h 
\sum_{i=k-1}^{h-k} q^{(k-1)^2} \qbi{i}{k-1}{q} \qbi{h-1-i}{k-1}{q}.$$  
The result is then obtained by shifting accordingly the indices $h$ and $i$.
\end{proof}

\subsection{Type $D$}
Finally, for type $D$, we have the following result.
\begin{proposition}\label{prop:majorD}
For any positive integer $n$, we have
\begin{equation}\label{cmajD}
\sum_{w\in \bar{D}_{n+1}^{FC}}q^{\maj(w)}=P_n(q)+(q^{2n+1}-q^n)\qbi{n-1}{\lfloor (n-1)/2 \rfloor}{q}+\qbi{n+1}{\lfloor (n+1)/2 \rfloor}{q},
\end{equation}
where for $n$ even, $$P_n(q):=\sum_{h=1}^{n-1}q^h\sum_{i=0}^{h-1}\qbi{h-1}{i}{q}+\frac{1}{2}q^n(1+q)\sum_{i=0}^{n-1}\qbi{n-1}{i}{q},$$
and for $n$ odd,
\begin{multline*}
P_n(q):=\sum_{h=1}^{n-1}q^h\sum_{i=0}^{h-1}\qbi{h-1}{i}{q}+\frac{1}{2}q^n(1+q)\sum_{i=0}^{n-1}\qbi{n-1}{i}{q}\\+\sum_{h=1}^{(n-1)/2}q^{n-h}\qbi{n-h-1}{(n-1)/2}{q}+\frac{1}{2}q^n(1+q)\qbi{n-1}{(n-1)/2}{q}.
\end{multline*}
\end{proposition}

\begin{proof}
Obviously from their definition, self-dual right-peaks of type $D_{n+1}$ are in bijection with the ones of type $B_n$, and have no descents at $s_n, s_{n+1}$, so  their generating function is equal to~\eqref{eq:LPmaj}.

The alternating involutions containing $s_n s_{n+1}$ are in bijection with walks in $\catsuffixset^*_n$ ending at height $1$. Among these, the walks having an ascending last step are enumerated by 
$$\displaystyle q^{2n+1}\qbi{n-1}{\lfloor (n-1)/2 \rfloor}{q}.$$ Indeed  both $s_n$ and $s_{n+1}$ are descents, and the first $n-1$ steps form a walk in $\catset^*_{n-1}$ so we can use Proposition~\ref{prop:majorA}. 

The generating function for the previous walks having a descending last step can be computed in the following way: first add an extra descending step to reach the $x$-axis, thus obtaining a walk in $\catset^*_{n+1}$ ending with two descending steps.  Note that neither $s_n$ nor $s_{n+1}$ are descents for the corresponding involutions, so the generating function is obtained by computing the generating function of all type $A$ walks from the origin to $(n+1,0)$, to which we must  substract the generating function of the walks having either a peak in position $(n,1)$ or ending with one horizontal step (see the figure below). 
\begin{center}
\includegraphics[width=12cm]{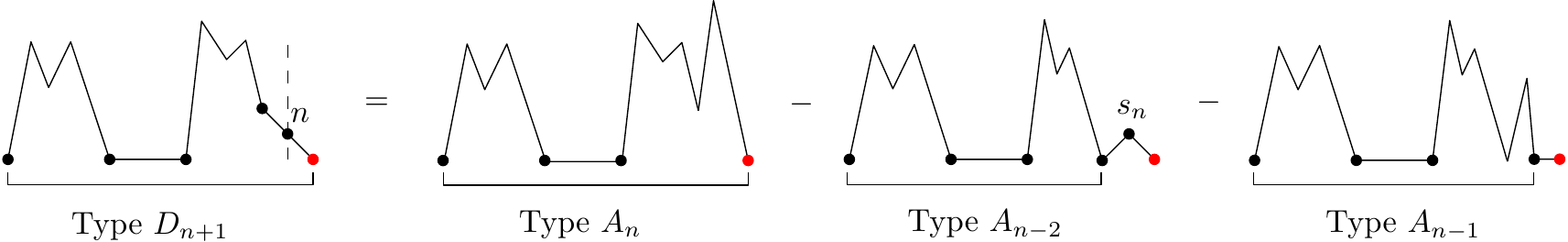}\label{fig:sn_no_typeD}
\end{center}
Again by Proposition~\ref{prop:majorA}, the result is
$$\qbi{n+1}{\lfloor (n+1)/2 \rfloor}{q} - q^n\qbi{n-1}{\lfloor (n-1)/2 \rfloor}{q} - \qbi{n}{\lfloor n/2 \rfloor}{q}.$$
By adding the two previous expressions to~\eqref{eq:LPmaj}, we get the two last terms in~\eqref{cmajD} minus $1$, therefore we now have  to show that the generating function for the remaining alternating elements is equal to $P_n(q)+1$.
Consider such an element, and let  $H$ be its corresponding heap. If there is no occurrence of $s_n$ and $s_{n+1}$, then $H$ is encoded by a walk in $\catset^*_{n}$. Otherwise notice that self-duality implies that  $|H_{\{s_n,s_{n+1}\}}|$ is odd. Therefore if the first label of $H_{\{s_n,s_{n+1}\}}$ is $s_n$ (\emph{resp.} $s_{n+1}$), then $H$ is in bijection with a walk $\gamma^{s_n}$ (\emph{resp.} $\gamma^{s_{n+1}}$)  in $\catsuffixset^*_n$ ending at an odd height. 

Denote by $\catset^*_{n}(q)$ the generating function~\eqref{cmajA} for type $A_{n-1}$, and 
$\catsuffixset^*_{s_n}(q)$ the generating function for the alternating FC involutions in $B_{n}$. Let also $\catsuffixset^*_{s_{n+1}}(q)$ be the generating function of alternating FC involutions in $B_n$ but where each occurrence of $s_n$ is replaced by $s_{n+1}$.

We claim that if $n$ is even, then the desired generating function is equal to 
\begin{equation}\label{eq:1}
 \catset^*_{n}(q) + \frac{1}{2} (\catsuffixset^*_{s_n}(q) -\catset^*_{n}(q)) + \frac{1}{2}( \catsuffixset^*_{s_{n+1}}(q)-\catset^*_{n}(q)).
\end{equation}
Recall that $\catsuffixset^*_{s_n}(q)$ counts walks in $\catsuffixset^*_n$. These walks can be split into three parts: the ones ending at $0$, at an even height $\geq 2$, or at an odd height. Then \eqref{eq:1} is  consequence of a bijection between these walks ending at an odd height, and those ending at an even $\geq 2$ height. This bijection goes as follows: in a walk ending at an odd height replace the last $(1,0)$ step (as $n$ is even there is at least one)  by a $(1,1)$ step, and shift accordingly the remaining steps. For the reverse application, change the last $(1,1)$ step starting from the $x$-axis into a $(1,0)$ step. This bijection preserves the peaks of the walks (and therefore the descents of the associated FC involutions).

Simplifying~\eqref{eq:1}, we have to compute $(\catsuffixset^*_{s_n}(q) +\catsuffixset^*_{s_{n+1}}(q))/2$, where thanks to the analysis of the previous type $B$ case,
$$\catsuffixset^*_{s_n}(q)=1+\sum_{h=1}^nq^h\sum_{i=0}^{h-1}\qbi{h-1}{i}{q}.$$
Notice that in type $B$, a FC involution has a descent at $s_n$ if and only if $a_1+b_1=n$, with the notations above. Therefore $\catsuffixset^*_{s_{n+1}}(q)$ can be computed similarly to $\catsuffixset^*_{s_{n}}(q)$, except for $h=n$, in which case the corresponding term has to be multiplied by $q$. This yields the desired expression $P_n(q)+1$.

If $n$ is odd, our previous bijection between walks ending at odd and even height does not work anymore, as our walks can have no horizontal step. So after splitting the walks in $\catsuffixset^*_n$ as before, but where the walks ending at an odd height are separated between those having or not an horizontal step, we can use the previous bijection to see that one must therefore replace~\eqref{eq:1} by 
\begin{equation*}\label{eq:2}
\frac{1}{2} (\catsuffixset^*_{s_n}(q)+ \catsuffixset^*_{s_{n+1}}(q))+\frac{1}{2} (\catsuffixset_{s_n}(q)+\catsuffixset_{s_{n+1}}(q)),
\end{equation*}
where $\catsuffixset_{s_n}(q)$ is the generating function for the walks in $\catsuffixset^*_n$ having no horizontal step, and $\catsuffixset_{s_{n+1}}(q)$ is the same generating function, but where each occurrence of $s_n$ is replaced by $s_{n+1}$. Note that  all these walks must end at an odd height. To conclude, it is enough to prove that 
$$\catsuffixset_{s_{n}}(q)=\sum_{h=0}^{(n-1)/2}q^{n-h}\qbi{n-h-1}{(n-1)/2}{q}.$$
To see this, consider any of the previous walks, which has, say  $j$, steps of the form $(1,1)$ leaving definitively a certain height (note that $j$ has to be odd, like $n$). Transform bijectively the $(j-1)/2$ first of these steps in the previous walk to steps $(1,-1)$ and shift accordingly all remaining steps. Remark that peaks are preserved by this transformation (this is a well known bijection, see \cite{BBES} and the references therein). Then transform our walk as in types $A$ and $B$,  to get a walk of length $n$ with paths $(1,0)$ and $(0,1)$, starting at the origin and ending at $((n-1)/2,(n+1)/2)$. The analysis is now the same as in type $A$, except that we have to distinguish whether the last step corresponds to a descent (if it is of the form $(0,1$)) or not. To do this, denote by $h$ the distance between the top extremity of the last vertical step and the point $((n-1)/2,(n+1)/2)$, as shown in the figure below, and then proceed as in type $A$ to relate our walks to integer partitions in a box.

\begin{center}
\includegraphics[width=10cm]{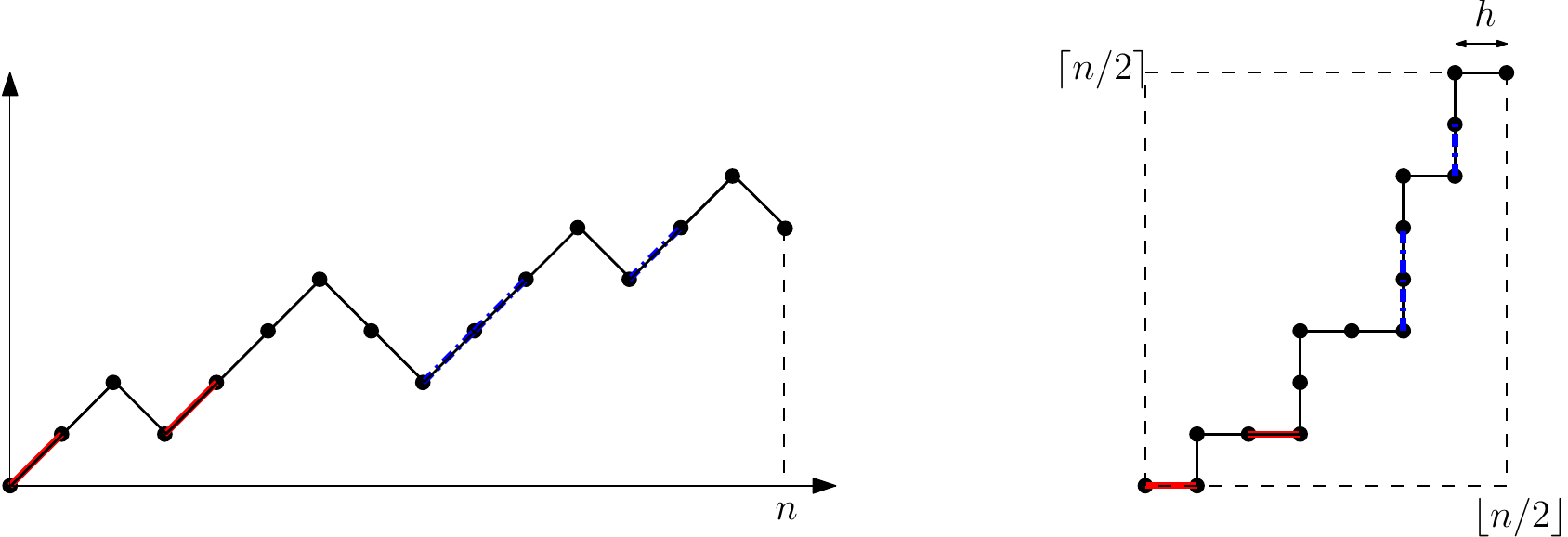}\label{fig:bijection_typeD}
\end{center}

\end{proof}

%%%%%%%%%%%%%%%%%%%%%%%%%%%%
\section{Enumeration with respect to the Coxeter length}\label{sec:length}
%%%%%%%%%%%%%%%%%%%%%%%%%%%%

%%%%%%%%%%%%%%%%%%%%%%%%%%%%
\subsection{Finite types}\label{subsec:lengthfinite}
%%%%%%%%%%%%%%%%%%%%%%%%%%%%

Recall from the introduction that for any Coxeter group $W$, the length generating function $\bar{W}^{FC}(t)$ is defined as $\sum_{w \in \bar{W}^{FC}} t^{\ell(w)}$. In the three classical finite types these series are polynomials, and we denote their generating functions by  
$$\bar{A}(x)=\sum_{n\geq1}\bar{A}_{n-1}^{FC}(t)x^{n}, \quad  \bar{B}(x)=\sum_{n\geq0}\bar{B}_{n}^{FC}(t)x^{n}, \ {\rm and} \quad  \bar{D}(x)=\sum_{n\geq0} \bar{D}_{n+1}^{FC}(t)x^{n},$$ where $\bar{D}_1^{FC}(t):=1$. 
\begin{proposition}\label{prop:involutions}
We have 
\begin{eqnarray}
\bar{A}(x)&=&\frac{\cat(x)}{1-x\cat(x)}-1,\label{typeAinv}\\
\bar{B}(x)&=&\frac{\catsuffix(x)}{1-x\cat(x)}+\frac{x^2t^3}{1-xt^2}\frac{\cat(x)\cat(tx)}{1-x\cat(x)},\label{typeBinv}\\
\bar{D}(x)&=& 2 \frac{\catsuffix^{o}(x)}{1-x\cat(x)} +\frac{\cat(x)}{1-x\cat(x)} +\frac{xt^2}{1-xt^2}\frac{\cat(x)\cat(tx)}{1-x\cat(x)},\label{typeDinv}
\end{eqnarray}
where $\catsuffix^{o}(x)$ denotes the generating function for  $t$-weighted prefixes of  Dyck walks ending at an odd height. Note that the previous generating functions can all be explicitly computed thanks to~\eqref{eqfunccat} and the following functional equations:
\begin{eqnarray}
\catsuffix(x)&=&\cat(x)(1+xt\catsuffix(tx)),\nonumber\\
\catsuffix^{o}(x)&=&xt\cat(x)\cat(tx)(1+xt^2\catsuffix^{o}(xt^2))\nonumber.
\end{eqnarray}
\end{proposition}

\begin{proof}
From Proposition~\ref{prop:typeclassiques}, and Proposition~\ref{proposition:walk_encoding}, FC involutions in $A_{n-1}$ are in one-to-one correspondence with walks in $\catset^*_n$. This shows~\eqref{typeAinv} by using~\eqref{lienDyck}.

From Proposition~\ref{prop:typeclassiques}, FC involutions in $B_n$ correspond to self-dual heaps which are either alternating or right-peaks. 

\begin{figure}[h!]
\begin{center}
\includegraphics[width=6cm]{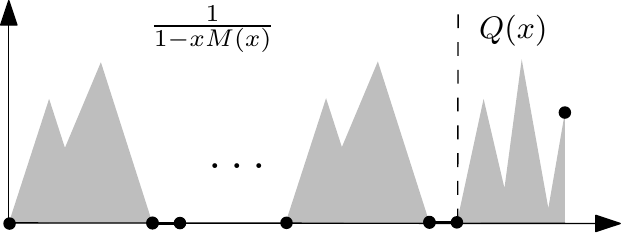}
\caption{Decomposition of walks for self-dual alternating heaps of type $B$.}\label{fig:Decomp_C+} 
\end{center}
\end{figure}
As before, in the alternating case, one can use Proposition~\ref{proposition:walk_encoding} to show that they are in one-to-one correspondence with walks in $\catsuffixset^*_n$. Moreover, such a walk can be seen as a collection of pairs (Dyck path, horizontal step), followed by a prefix  of a Dyck path. This is illustrated in Figure~\ref{fig:Decomp_C+} and shows that
$$\catsuffix^*(x)=\frac{\catsuffix(x)}{1-x\cat(x)},$$
yielding the first term on the right-hand side of~\eqref{typeBinv}.  

By definition right-peaks can be split into two parts $H_{\{s_1,\ldots, \dot{s}_{j}\}}$ and $H_{\{s_j,\ldots, s_n\}}$, and this decomposition is unique by  Remark~\ref{rem:familyPeaks}. The generating function of the second part is  
$$\frac{x^2t^3}{1-xt^2},$$
 while by Proposition~\ref{proposition:walk_encoding} the first part is in bijection with walks in $\catsuffixset^*_n$ ending at height $1$. The latter can be seen as collections of pairs (Dyck path, horizontal step),  followed by a couple of Dyck walks, the first  starting and ending at height 0, and the second at height 1, as illustrated below.

\begin{figure}[h!]
\begin{center}
\includegraphics[width=8cm]{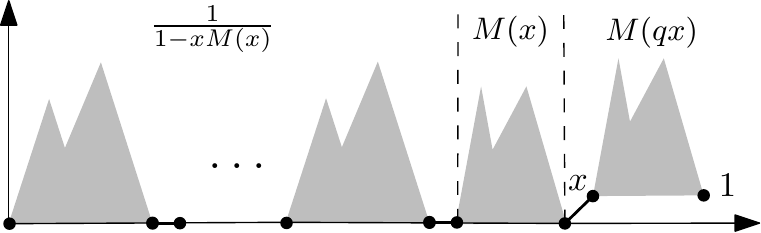}
\caption{Decomposition of walks for self-dual right-peaks of type $B$.}\label{fig:Decomp_x}
\end{center}
\end{figure}

This yields the generating function $$\frac{xt \cat(x)\cat(tx)}{1-x\cat(x)}.$$ The second term on the right-hand side of~\eqref{typeBinv} arises as the product of these two generating functions, after dividing by $xt$, which corresponds to the occurrence of one $s_j$ belonging to both parts of our split.

Finally, for the proof of~\eqref{typeDinv}, we use the type $D_{n+1}$ case of Proposition~\ref{prop:typeclassiques}. Starting from a self-dual alternating heap $H$ of type $B_n$, we have three cases:
\begin{itemize}
\item if there is no occurrence of $s_n$ in $H$, then Proposition~\ref{proposition:walk_encoding} and formula~\eqref{lienDyck} give the generating function $\cat(x)/(1-x\cat(x))$.
\item if there is at least one occurrence of $s_n$ in $H$ (self-duality implies this number is odd), we replace $H_{s_n}$ by $(s_n,s_{n+1},\ldots,s_n)$ or $(s_{n+1}, s_n,\ldots,s_{n+1})$, giving rise to two heaps corresponding to elements of $\bar{D}_{n+1}^{FC}$, which are in bijection with walks in $\catsuffix^*_n$ ending at odd height. 
This gives the generating function $2 \catsuffix^o(x)/(1-x\cat(x))$, with the same decomposition as in 
Figure~\ref{fig:Decomp_C+}, in which $\catsuffix(x)$ is replaced by $\catsuffix^o(x)$.
\item if there is exactly one occurrence of $s_n$ in $H$, we must add the generating function for the self-dual FC heaps corresponding to type $D_{n+1}$ involutions, by replacing $s_n$ by $s_ns_{n+1}~(=s_{n+1}s_n)$. These elements are in bijection with walks in $\catsuffix^*_n$ ending at height $1$, whose generating function has already been computed (see Figure~\ref{fig:Decomp_x}). Taking into account an extra $t$ due to the occurrence of $s_ns_{n+1}$, 
we obtain the generating function $xt^2 \cat(x)\cat(tx)/(1-x\cat(x))$.
\end{itemize}

The generating function of right-peaks of type $D_{n+1}$ is the one for type $B_n$ multiplied by $t$, and it is equal to 
$$\frac{x^2t^4}{1-xt^2}\frac{\cat(x)\cat(tx)}{1-x\cat(x)}.$$
By summing up these four expressions we obtain formula~\eqref{typeDinv} after a few simplifications. Finally the functional equations for $\catsuffix(x)$ and $\catsuffix^o(x)$ can be computed thanks to classical decompositions of (prefixes of) Dyck walks.
\end{proof}

%%%%%%%%%%%%%%%%%%%%%%%%%%%%%%%%%%%%%
\subsection{Affine type $\aff{A}_{n-1}$ }\label{sub:GFatilde}
%%%%%%%%%%%%%%%%%%%%%%%%%%%%%%%%%%%%%

In this section, we will give the description of the heaps for FC involutions in $\aff{A}_{n-1}$, and  their generating function with respect to the length. 

%Moreover, we will use our classification to describe the connection with the cells in the associated Temperley--Lieb algebra as defined in~\cite{FanGreen_Affine}.

Recall from~\cite[Proposition~2.1]{BJN} that an element $w$ in $\Aaff_{n-1}$ is FC if and only if, in  any reduced decomposition of $w$, the occurrences of $s_i$ and $s_{i+1}$ alternate for all $i\in \{0,\ldots,n-1\}$, where we set $s_n=s_0$. This shows that FC heaps for the graph of type $\Aaff_{n-1}$ coincide exactly with alternating heaps (extended to the non-linear Coxeter diagram of $\aff{A}_{n-1}$).

Moreover, from Lemma~\ref{lemme:heapinvolution}, we immediately derive the following description.
\begin{proposition}\label{prop:descriptionAaffine}
A FC element $w\in\aff{A}_{n-1}$ is an involution if and only if $\H(w)$ is a self-dual alternating heap of type $\aff{A}_{n-1}$ .
\end{proposition}

We shall represent heaps of type $\aff{A}_{n-1}$ by depicting all (alternating) chains $H_{\{s_i,s_{i+1}\}}$ for $i=0,\ldots,n-1$. To be able to represent these chains in a planar fashion, we duplicate the set of $s_0$-elements and use one copy for the depiction of the chain $H_{\{s_0,s_{1}\}}$ and one copy for $H_{\{s_{n-1},s_0\}}$. This can be seen in Figure~\ref{fig:bijection_typeA}, where we also illustrate the restriction of the bijection with walks from~\cite{BJN} to self-dual alternating heaps corresponding to involutions.

\begin{figure}[h!]
\begin{center}
\includegraphics[width=12cm]{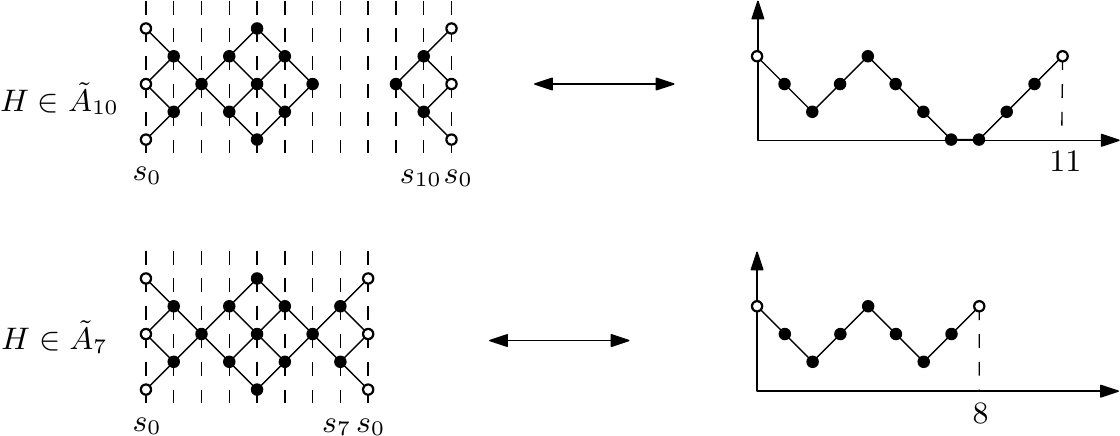}%\label{fig:Atilde_invol}
\end{center}
\caption{Two illustrations of the bijection between self-dual alternating heaps and walks in type $\aff{A}$.}\label{fig:Atilde_invol}
\end{figure}
Note that Proposition~\ref{prop:descriptionAaffine} immediately implies finiteness of the set of FC involutions in $\aff{A}_{n-1}$ if $n$ is odd, as the corresponding heaps $H$ can not have \emph{full support}, i.e. there exists necessarily an integer $i$ such that $|H_{s_i}|=0$. Indeed, thanks to self-duality and the alternating condition, one can see that there exists even in this case an integer $i$ such that $|H_{\{s_i, s_{i+1}\}}|=0$. Moreover, we are able to deduce the following result, where as usual $[x^n] F(x)$ stands for the coefficient of $x^n$ in the power series $F(x)$.

\begin{proposition}\label{prop:involutionsAaffines}
The generating function, with respect to the length, of fully commutative involutions of type $\aff{A}_{n-1}$, is equal to
\begin{equation}\label{eq:AaffineInv}
\frac{t^{n}\touch{\Cyl}_n(t)}{1-t^{n}}+[x^n]\frac{\cat(x)(1+tx^2\frac{\partial(x\cat)}{\partial x}(tx))}{1-x\cat(x)},
\end{equation}
where $\Cyl_n(t)$ is the generating function for walks in $\gensetdyck_n$  starting and ending at the same height. Moreover, if $n$ is odd,~\eqref{eq:AaffineInv} is a polynomial, while the corresponding growth sequence is ultimately periodic with period dividing $n$, if $n$ is even. Finally, in the latter case, periodicity starts at length $1+n^2/4$.
\end{proposition}

\begin{proof}
From Proposition~\ref{prop:descriptionAaffine},  we know that FC involutions in $\aff{A}_{n-1}$ correspond to  self-dual alternating heaps. Next, by using the restriction to self-dual heaps of the bijection $\varphi'$ of~\cite[Theorem 2.2]{BJN}, we see that these involutions are in one-to-one correspondence with walks in $\genset_n$  starting and ending at the same height.  Moreover, the length of the corresponding involution is equal to the area below the walk. Such walks can be decomposed according to whether they hit the $x$-axis or not, as can be seen below.

\begin{center}
\includegraphics[width=10cm]{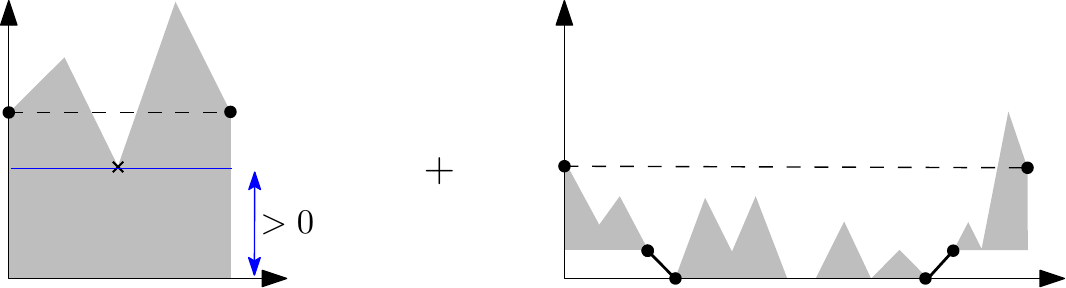}\label{fig:Decomp_Atilde_invol}
\end{center}
The first term in formula~\eqref{eq:AaffineInv} corresponds to the walks on the left. For the walks on the right, the central part is  in $\catset^*_j$ (for a $j$ between $0$ and $n-2$) whose generating function is given by~\eqref{lienDyck}, while the other two parts can be joined to form a path in $\catset_{n-j-2}$. The second term in formula~\eqref{eq:AaffineInv} is therefore obtained through classical considerations (see for example the proof of~\cite[Corollary 2.4]{BJN}).

Next, it is clear that if $n$ is odd, then  $\Cyl_n(t)=0$, while ultimate periodicity is obvious from formula~\eqref{eq:AaffineInv} if $n$ is even. Finally, in the latter case, the second term in~\eqref{eq:AaffineInv} is a polynomial with degree $n^2/4$, which corresponds to the area of the walk starting and ending at the origin (having therefore $n/2$ up steps followed by $n/2$ down steps), and this finishes the proof.
\end{proof}

%%%%%%%%%%%%%%%%%%%%%%%%%%%%%%%%%%%%%%%%%%
\subsection{Other classical affine types}\label{subsec:invaffBCD}
%%%%%%%%%%%%%%%%%%%%%%%%%%%%%%%%%%%%%%%%%%
By using the classification results of~\cite{BJN}, it is possible to compute all the generating functions for FC involutions $\bar{W}^{FC}(t)$ when $W$ is affine. Nevertheless, we will only give the results for types $\aff{C}$, $\aff{B}$  and $\aff{D}$, and we leave the exceptional types to the interested reader. In the following result, we give the general forms of these generating functions, which shows ultimate periodicity of the growth sequence in all cases. Note that the exact beginning of periodicity could also be computed by our methods.

Denote by  $\mathcal{F}^e_n$, (\emph{resp.} $\mathcal{F}^o_n$) the subfamily of $\mathcal{G}_n$ of paths ending at an even (\emph{resp.} odd) height. Denote also by $\mathcal{F}^{ee}_n$, (\emph{resp.} $\mathcal{F}^{oo}_n$) the subfamily of $\mathcal{G}_n$ of paths starting and ending at an even (\emph{resp.} odd) height. Recall that to each one of these families we associate a generating function as explained in Section~\ref{subsec:walks}. 

\begin{proposition}\label{prop:involutionsBCDaffines}
The generating functions for fully commutative involutions in types $\aff{C}_n$, $\aff{B}_{n+1}$  and $\aff{D}_{n+2}$ are respectively given by:
\begin{eqnarray}
&&\frac{t^{n+1}\touch{F}_n(t)}{1-t^{n+1}}+\frac{2t^{2n+3}}{1-t^2}+R_n(t)\label{eq:CaffineInv},\\
&&\frac{2t^{2n+2}\touch{F}^{o}_n(t)+2t^{n+1}\touch{F}^{e}_n(t)}{1-t^{2n+2}}+\frac{t^{2n+4}}{1-t}+\frac{t^{4n+2}}{1-t^{2n+1}}+T_n(t)\label{eq:BaffineInv},\\
&&\frac{4t^{2n+2}\touch{F}^{oo}_n(t)+4t^{n+1}\touch{F}^{ee}_n(t)}{1-t^{2n+2}}+\frac{2t^{2n+6}}{1-t^{2}}+\frac{2t^{4n+4}}{1-t^{2n+2}}+U_n(t),\label{eq:DaffineInv}
\end{eqnarray}
 where $R_n(t), T_n(t)$, and $U_n(t)$ are polynomials.
The corresponding growth sequences are ultimately periodic,  with period dividing $2n+2$, $(2n+1)(2n+2)$, and $2n+2$, respectively. 
\end{proposition}

\begin{proof}
From~\cite[Theorem 3.4]{BJN} and Lemma~\ref{lemme:heapinvolution}, FC involutions in $\aff{C}_{n}$ correspond to self-dual heaps belonging to five families, among which only two are infinite. It is therefore enough to focus on these two sets, the first of which being made of alternating heaps. From Proposition~\ref{proposition:walk_encoding}, such elements correspond to walks in $\genset_n$; their generating function is given by
$$\frac{t^{n+1}\touch{F}_n(t)}{1-t^{n+1}}+[x^n]\frac{(1+tx\catsuffix(tx))^2}{1-x\cat(x)},$$
by considering as before the decomposition illustrated below.

\begin{figure}[h!]
\begin{center}
\includegraphics[width=12cm]{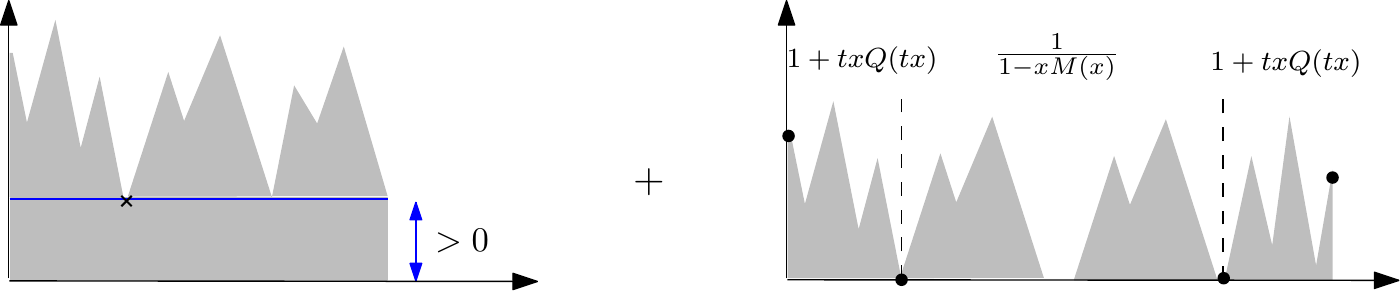}
\end{center}
\caption{Decomposition of walks in $\genset_{n}$.}\label{fig:Decomp_C_affine}
\end{figure}

The elements of the second infinite set are  the self-dual zigzags; they are  self-dual heaps of the form $H=\H(\mathbf{w})$ where $\mathbf{w}$ is a finite factor of the infinite periodic word $\left(u_ns_1s_2\cdots s_{n-1}s_ns_{n-1}\cdots s_2s_1\right)^\infty$ such that $|H_{s_i}| \geq 4$ for at least one $i\in\{1,\dots,n-1\}$.  The minimal length  is equal to $2n+3$, as can be seen in Figure~\ref{fig:Zigzag}, left. Moreover,  there are $2$ self-dual heaps  for each length $2n+3+2k$ (with $k\geq0$) by left-right symmetry, and $0$ for other lengths, yielding the generating function $$\displaystyle \frac{2t^{2n+3}}{1-t^2}.$$
Summing these  expressions and adding the polynomial generating functions  corresponding to self-dual heaps in the three remaining finite families from~\cite{BJN} gives~\eqref{eq:CaffineInv}.

For FC involutions of type $\aff{B}_{n+1}$, and from~\cite{BJN} and Lemma~\ref{lemme:heapinvolution}, it suffices to examine among type $\aff{B}_{n+1}$ alternating and zigzag heaps which ones are self-dual. Concerning alternating ones, we perform the following substitution in each self-dual alternating heap $H$ of type $\aff{C}_{n}$: replace the occurrence of  $(s_n,s_n,s_n,\ldots)$ by $(s_n,s_{n+1},s_n,\ldots)$ or $(s_{n+1},s_n,s_{n+1},\ldots)$, together with the replacement $s_n\mapsto s_n s_{n+1}=s_{n+1}s_n$ in the case of exactly one occurrence of $s_n$. The self-duality condition forces $|H_{s_n}|$ to be either odd, or equal to $0$. Therefore the generating function for self-dual alternating heaps of type $\aff{B}_{n+1}$ takes the form
$$2\Geno_n(t)-[x^n]\frac{Q(x)}{1-xM(x)} +[x^n]\left(t^2\catsuffix(tx)+\frac{xt^2\catsuffix(x)\cat(tx)}{1-x\cat(x)}\right),$$
where $\Geno_n(t)$ is the generating function for walks in $\genset_n$ ending at odd height. Moreover, the only non-polynomial term in this expression is $\Geno_n(t)$. The corresponding walks can be decomposed as in Figure~\ref{fig:Decomp_C_affine} in which the parity of $h$ has to be taken into account, and the last term $1+tx \catsuffix(tx)$  has to be replaced by $\catsuffix^o(x)$. This yields 
$$\Geno_n(t)=\frac{t^{2n+2}\touch{F}^{o}_n(t)}{1-t^{2n+2}}+\frac{t^{n+1}\touch{F}^{e}_n(t)}{1-t^{2n+2}}+[x^n]\frac{(1+tx \catsuffix(tx))\catsuffix^o(x)}{1-x\cat(x)}.$$

\begin{figure}[h!]
\begin{center}
\includegraphics[width=12cm]{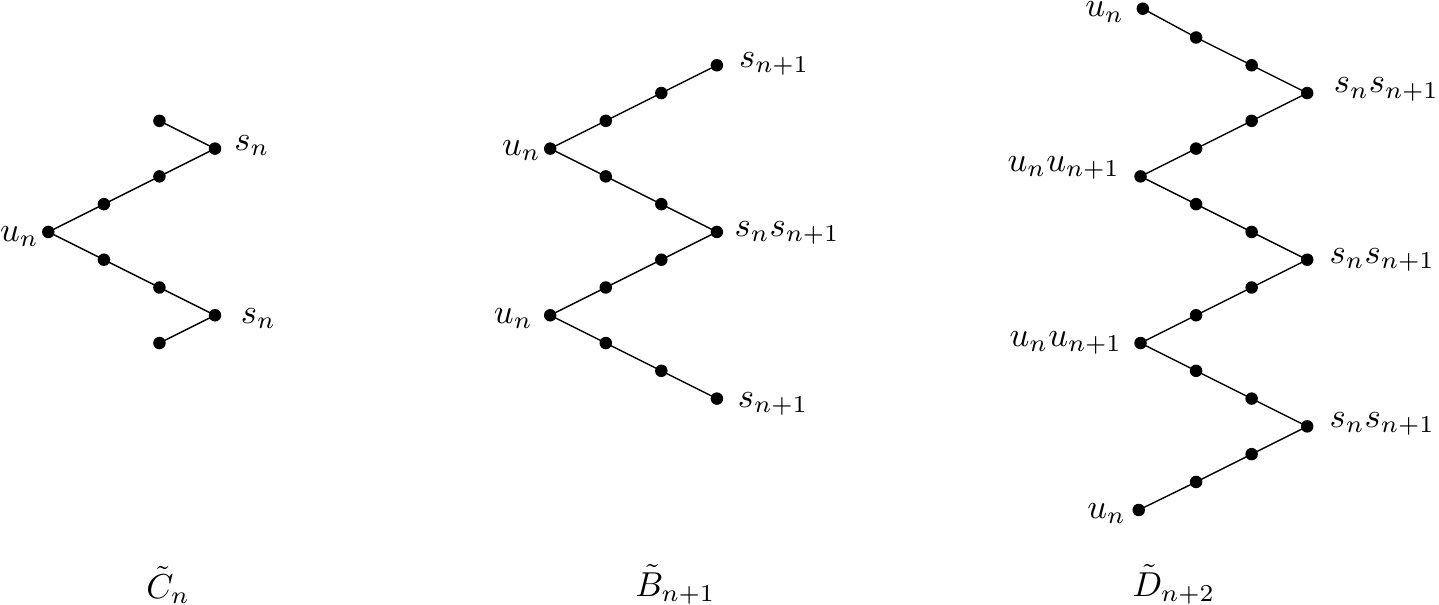}
\end{center}
\caption{Some self-dual zigzag heaps.}\label{fig:Zigzag}
\end{figure}

The minimal length for self-dual zigzags of type $\aff{B}_{n+1}$ is equal to $2n+4$. Moreover, for any $k\geq 0$ there are $2$ self-dual heaps for each length $4n+2+k(2n+1)$: these are the heaps starting and ending at $s_n$ or $s_{n+1}$ (see Figure~\ref{fig:Zigzag}, middle). For other lengths there is only $1$ such heap, yielding the generating function 
$$\displaystyle \frac{t^{2n+4}}{1-t}+\frac{t^{4n+2}}{1-t^{2n+1}}.$$
 It is then easy to derive formula~\eqref{eq:BaffineInv} by adding the previous generating functions to the polynomials corresponding to self-dual heaps in the three remaining (finite) families of~\cite{BJN}.

Finally, for FC involutions of type $\aff{D}_{n+2}$, we again consider the self-dual heaps among alternating and zigzag ones.  The alternating self-dual heaps correspond to walks  in $\genset_n$ starting and ending at odd height. Moreover, each such walk gives rise to $4$ FC involutions, according to the choices of $u_n$ or $u_{n+1}$, and $s_n$ or $s_{n+1}$. Among these walks, only the ones staying above the $x$-axis correspond to an infinite number of elements, their generating function is given by:
$$\frac{4t^{2n+2}\touch{F}^{oo}_n(t)+4t^{n+1}\touch{F}^{ee}_n(t)}{1-t^{2n+2}},$$
as can be seen as before through classical walk decomposition. 
The minimal length for self-dual zigzags is $2n+6$. For any $k\geq 0$,  by left-right symmetry there are $4$ such elements for each length $4n+4+k(2n+2)$ (see Figure~\ref{fig:Zigzag}, right), $2$ for other even lengths, and $0$ for odd lengths, yielding the generating function 
$$\displaystyle \frac{2t^{2n+6}}{1-t^2}+\frac{2t^{4n+4}}{1-t^{2n+2}}.$$
\end{proof}

%%%%%%%%%%%%%%%%%%%%%%%%%%%%%%%%%%%%%
\section{Cells for FC elements in type~$\aff{A}$}
\label{sec:cells}
%%%%%%%%%%%%%%%%%%%%%%%%%%%%%%%%%%%%%
This section is not about enumeration, but aims at illustrating how the representation of FC elements  as heaps can be useful in other ways.

In~\cite{FanGreen_Affine}, Fan and Green study the affine
Temperley--Lieb algebra ${\rm TL}(\aff{A}_{n-1})$. It is a quotient of
the type $\aff{A}_{n-1}$ Hecke algebra, and can be defined by generators $E_{s_i}$ for
$i\in \{0,1,\ldots,{n-1}\}$ and relations:
 \[\begin{cases}
 E_{s_i}^2=E_{s_i},\\
 E_{s_i}E_{s_{j}}E_{s_i}=E_{s_i}\quad\text{if $i= j\pm 1$ modulo $n$,}\\
 E_{s_i}E_{s_j}=E_{s_j}E_{s_i}\quad\text{if $i\neq j\pm 1$ modulo $n$.}
\end{cases}
\]
Note that the first relation involves usually an extra
parameter $\alpha$, but this has no incidence on the results we will describe.
The algebra ${\rm TL}(\aff{A}_{n-1})$ has a linear basis $(E_w)$ indexed by {\rm FC} elements in $\aff{A}_{n-1}$: one can define unambiguously $E_w=E_{s_{i_1}}\cdots
E_{s_{i_k}}$ where $s_{i_1}\cdots s_{i_k}$ is any reduced expression
of the FC element $w$.

Using this algebra, there are natural relations on the set of FC elements. 

\begin{definition}\label{defi:cells}
Let $w,w'$ be FC elements of type $\aff{A}_{n-1}$. One writes
$w\stackrel{R}{\leqslant} w'$ if there exists a {\rm FC} element $x$ such that
$E_{w'}=E_wE_x$, and $w\stackrel{R}{\sim} w'$ if
$w\stackrel{R}{\leqslant} w'$ and $w'\stackrel{R}{\leqslant} w$.
\end{definition}

Since $\stackrel{R}{\leqslant}$ is a preorder, $\stackrel{R}{\sim}$ is
an equivalence relation whose classes are called {\em right cells}. These are analogues of the famous Kazhdan--Lusztig cells which give representations of the Hecke algebras, in the arguably simpler context of the Temperley--Lieb  algebra ${\rm TL}(\aff{A}_{n-1})$.

\begin{theorem}\cite[Theorem 3.5.1]{FanGreen_Affine}
\label{theo:cells_involutions}
Each right cell contains at most one involution. Right cells with no
involution occur only when $n$ is even.
\end{theorem}

In the sequel we wish to show how the use of heaps can illustrate this result and make it more precise. 
To this aim we first describe in terms of heaps the so-called {\em reduction} of {\rm FC} elements used in \cite{FanGreen_Affine}.

\begin{definition}[Reduction]\label{defi:reduction}
Let $w$ be a FC element of type $\aff{A}_{n-1}$, and $s_i\in
{\rm Des}_R(w)$. Then $w$ can be reduced to $ws_i$ if at least one of
$s_{i-1},s_{i+1}$ belongs to ${\rm Des}_R(ws_i)$. We will write
$w\stackrel{R}{\rightarrow} ws_i$.
\end{definition}

Reduction is easy to illustrate on heaps: $w\stackrel{R}{\rightarrow}
ws_i$ if $s_i$ labels a maximal element in $\H(w)$ and if either
$s_{i-1}$ or $s_{i+1}$ labels a maximal element in $\H(ws_i):=\H(w)\setminus\{s_i^{top}\}$, where $s_i^{top}$ is the maximal element of the chain $H_{s_i}$. We refer the reader to  Figure~\ref{fig:irr1} for a chain of successive reductions.

\begin{figure}[!ht]
\begin{center}
\includegraphics[width=11cm]{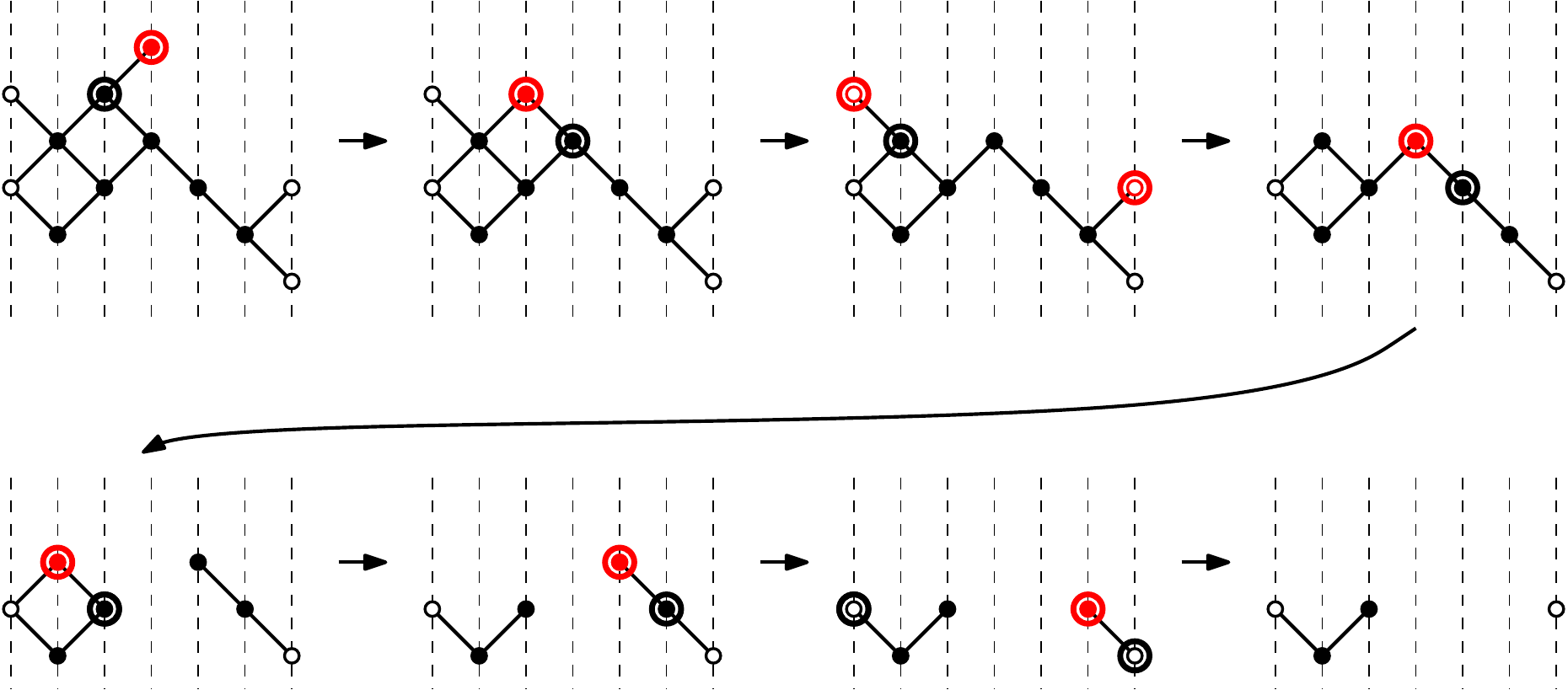}
\caption{Successive reductions. \label{fig:irr1}}
\end{center}
\end{figure}

Reduction is useful to investigate right cells, thanks to the following simple result.
\begin{lemma}
\label{lemma:reduction_cells}
If $w\stackrel{R}{\rightarrow} ws_i$, then $w\stackrel{R}{\sim} ws_i$.
In other words, if $w$ reduces to $ws_i$, then both belong to the same right
cell.
\end{lemma}

A FC element $w$ is called {\em irreducible} if it can not be
reduced to  $ws_i$ for any $i$. It is relatively easy to give a
characterization of the heaps of such elements. Recall that the
support of the FC element $w$ is the set of $s_i$, $i\in\{0,\ldots,n-1\}$, which occur in a
reduced decomposition of $w$. 
%We write $I_1,\ldots,I_k$ for the
%maximal cyclic intervals which form the support.

\begin{proposition}
 A FC element $w$ is irreducible if and only if its heap 
satisfies $s_i^{top}>s_{i+1}^{top}<s_{i+2}^{top}>s_{i+3}^{top}<\ldots
s_{i+2m-1}^{top}<s_{i+2m}^{top}$ for all $i,m$ satisfying:
\begin{itemize}
\item if $w$ has full support, then  $n$ is
even, $m=n/2$ and $i=0$ or $1$, 
\item otherwise $\{i,i+1,\ldots,i+2m\}$ is any maximal (cyclic) interval of the support
of $w$.
\end{itemize}
\end{proposition}

\begin{proof}
First it is clear that such elements are indeed irreducible.
Suppose $w$ is not of this form. Then (up to a cyclic shift of the
indices and a relabeling $i\mapsto n-i$), the following situation
occurs: one has $s_0^{top}>s_1^{top}$ and either $|H_{s_2}|=0$ or
$s_1^{top}>s_2^{top}$. Therefore $w\stackrel{R}{\rightarrow} s_0w$ and
we are done.
\end{proof}
For $w$ irreducible, we now define a particular subset of elements in $\H(w)$: this will allow us to connect irreducible heaps with FC involutions. 

 If the support of $w$ is not full, then we select all maximal elements in the heap. This is illustrated in Figure~\ref{fig:irr2}, left. 
\begin{figure}[!ht]
\begin{center}
\includegraphics[width=12cm]{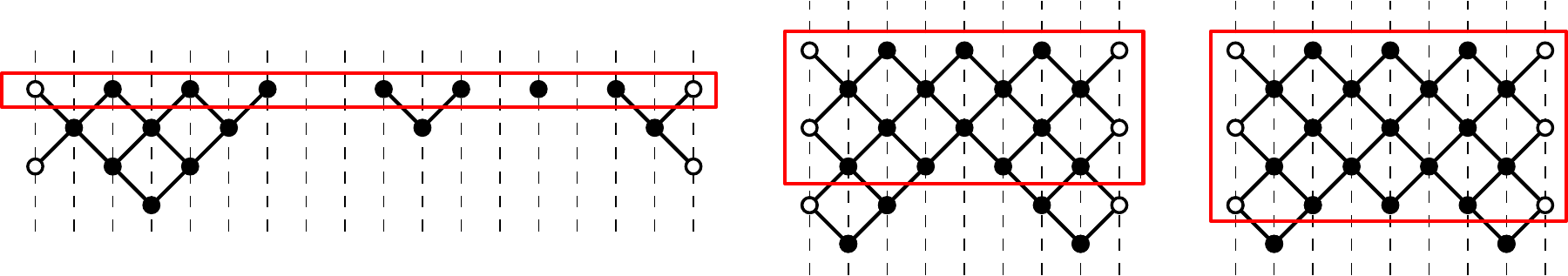}
\caption{Heaps of irreducible elements.\label{fig:irr2}}
\end{center}
\end{figure}

\noindent Otherwise, $n$ is even by the proposition. Define $R_0:=s_0s_2\cdots s_{n-2}$ and $R_1:=s_1s_3\cdots s_{n-1}$. Then select in $\H(w)$ the upper ideal isomorphic to the heap of $R_\epsilon R_{1-\epsilon}R_\epsilon R_{1-\epsilon}\cdots R_\delta R_{1-\delta}$ with the maximal number of factors, where $\epsilon,\delta\in\{0,1\}$. Two such examples are illustrated in Figure~\ref{fig:irr2}, middle and right, the one in the middle (\emph{resp.} right) having an odd (\emph{resp.} even) number of such factors.

In each case, denote by $H^{top}$ the subset of selected elements and $H^{bottom}$ the complement in $\H(w)$.

\begin{proposition}\cite{FanGreen_Affine}
Distinct irreducible elements belong to different right cells.
\end{proposition}

We will not prove this proposition which is arguably the crucial part
of the argument in~\cite{FanGreen_Affine}. An immediate consequence is
that \emph{each right cell contains precisely one irreducible element}.

To obtain Theorem~\ref{theo:cells_involutions}, we need to relate
irreducible elements to involutions. One easily constructs an irreducible element from a FC involution by reducing it repeatedly. It can be seen in this case that the
process is reversible: given such an irreducible element with heap $H$, take the dual of $H^{bottom}$ and add it as an upper ideal to $H$, as illustrated in Figure~\ref{fig:irr3}. 

\begin{figure}[!ht]
\begin{center}
\includegraphics[width=11cm]{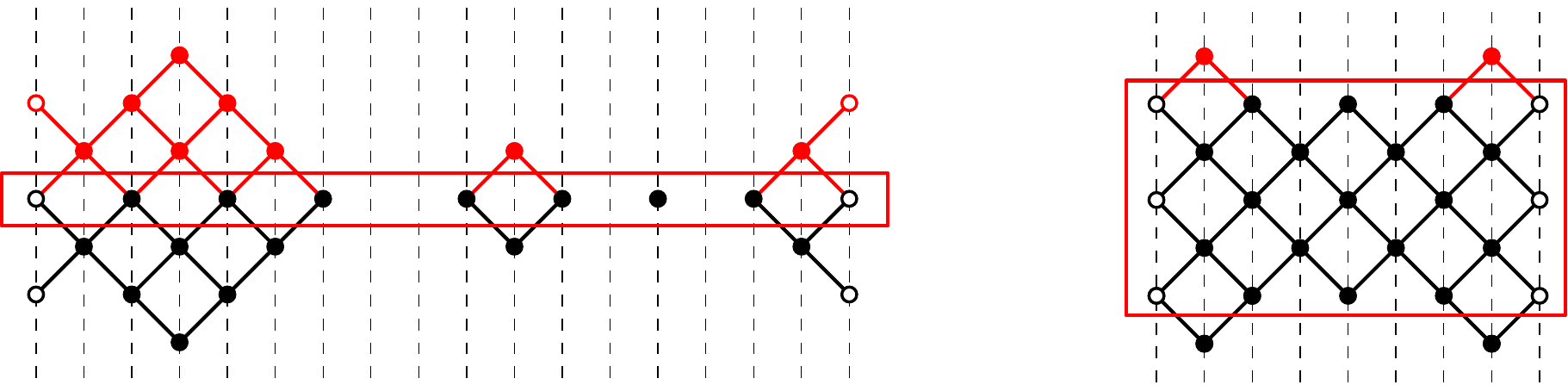}
\caption{Heaps of FC involutions corresponding to irreducible elements.\label{fig:irr3}}
\end{center}
\end{figure}

Right cells with no involution are now easy to characterize: those are the ones whose unique irreducible element has full support and a top part $H^{top}$ with an even number of factors $R_0,R_1$: see for example the right heap in Figure~\ref{fig:irr2}. Indeed the inverse process does not produce a FC heap in this case.

%%\nocite{*}
%\bibliographystyle{plain}
%\bibliography{fullycommut}

\def\cprime{$'$}

\end{document}